\pdfoutput=1
\RequirePackage{ifpdf}
\ifpdf 
\documentclass[pdftex]{sigma}
\else
\documentclass{sigma}
\fi

\usepackage{longtable}
\usepackage{pdflscape}

\usepackage[skip=5pt,labelfont=bf,labelsep=period]{caption}

\numberwithin{equation}{section}

\newtheorem{Theorem}{Theorem}[section]

\newtheorem{Lemma}[Theorem]{Lemma}
\newtheorem{Proposition}[Theorem]{Proposition}
 { \theoremstyle{definition}
\newtheorem{Definition}[Theorem]{Definition}
\newtheorem{Question}[Theorem]{Question}
\newtheorem*{Notation}{Notation}

\newtheorem{Remark}[Theorem]{Remark} }

\DeclareMathOperator{\Ad}{Ad}
\DeclareMathOperator{\tr}{tr}
\DeclareMathOperator{\ad}{ad}
\DeclareMathOperator{\Span}{span}
\DeclareMathOperator{\Lie}{Lie}
\DeclareMathOperator{\id}{id}
\DeclareMathOperator{\Ric}{Ric}
\DeclareMathOperator{\Diag}{diag}

\begin{document}
\allowdisplaybreaks

\newcommand{\arXivNumber}{2107.11455}

\renewcommand{\PaperNumber}{109}

\FirstPageHeading

\ShortArticleName{Scalar Curvatures of Invariant Almost Hermitian Structures}

\ArticleName{Scalar Curvatures of Invariant Almost Hermitian\\ Structures on Generalized Flag Manifolds}

\Author{Lino GRAMA~$^{\rm a}$ and Ailton R.~OLIVEIRA~$^{\rm b}$}

\AuthorNameForHeading{L.~Grama and A.R.~Oliveira}

\Address{$^{\rm a)}$~IMECC - Universidade Estadual de Campinas (Unicamp), Departamento de Matem\'{a}tica,\\
\hphantom{$^{\rm a)}$}~Rua S\'{e}rgio Buarque de Holanda,
651, Cidade Universit\'{a}ria Zeferino Vaz.\\
\hphantom{$^{\rm a)}$}~13083-859 Campinas - SP, Brazil}
\EmailD{\href{mailto:linograma@gmail.com}{linograma@gmail.com}}
\URLaddressD{\url{https://sites.google.com/unicamp.br/linograma/}}

\Address{$^{\rm b)}$~UEMS - Universidade Estadual de Mato Grosso do Sul - MS, Cidade Universit\'aria\\
\hphantom{$^{\rm b)}$}~de Dourados, Rodovia Itahum, Km 12 s/n - Jardim Aeroporto, Dourados - MS, Brazil}
\EmailD{\href{mailto:ailton_rol@yahoo.com.br}{ailton\_rol@yahoo.com.br}}

\ArticleDates{Received August 02, 2021, in final form December 11, 2021; Published online December 21, 2021}

\Abstract{In this paper we study invariant almost Hermitian geometry on generalized flag manifolds. We will focus on providing examples of {\em K\"ahler like scalar curvature metric}, that is, almost Hermitian structures $(g,J)$ satisfying $s=2s_{\rm C}$, where $s$ is Riemannian scalar curvature and $s_{\rm C}$ is the Chern scalar curvature.}


\Classification{53C55; 53C21; 14M15}

\section{Introduction}

In the present paper we will study curvature properties of invariant almost Hermitian structures on homogeneous spaces. Special attention will be given to providing examples of {\em K\"ahler like scalar curvature $($Klsc$)$ metric}, that is, almost Hermitian structures $(g,J)$ satisfying $s=2s_{\rm C}$, where $s$ is the Riemannian scalar curvature and $s_{\rm C}$ is the Chern scalar curvature on generalized flag manifolds.

Recall the geometric meaning of the Riemannian scalar curvature $s$: given a Riemannian manifold $(M,g)$, the volume of the geodesic ball of radius $r$ with center at $p\in M$ has the asymptotic expansion as follows:
\begin{gather}\label{ee1}
\operatorname{vol}( B(p,r))=\omega_{n}r^{n}\bigg(1-\dfrac{s(p)}{6(n+2)}r^{2}+\mathcal{O}\big(r^{4}\big)\bigg),
\end{gather}
where $\omega_{n}$ is the volume of the unity ball in $\mathbb{R}^{n}$, see \cite{besse,lock}. Therefore the Riemannian scalar curvature $s(p)$ is positive or negative at a point $p$, if the volume of a small geodesic ball at~$p$ is respectively smaller or larger than the corresponding Euclidean ball of the same radius. If~the metric is not K\"ahler, there is not an immediate geometric interpretation for the Chern scalar curvature $s_{\rm C}$. However if the Hermitian structure is non-K\"ahler but satisfies $s=2s_{\rm C}$, a curvature~$s_{\rm C}$ has an interpretation like $s$ in the expression \eqref{ee1}.

Dabkowski and Lock in \cite{lock} call the Hermitian metrics satisfying the equation
\begin{gather*}
s=2s_{\rm C}
\end{gather*}
by {\it K\"ahler like scalar curvature $($Klsc$)$ metric}, and the authors exhibit examples of non-compact Hermitian manifolds satisfying $s=2s_{\rm C}$.

According to~\cite{apo} and~\cite{integra} nearly-K\"ahler manifolds and almost-K\"ahler manifolds satisfying $2s_{\rm C}-s=0$ are K\"ahler. Fu and Zhou show in \cite{fu} that if the pair $(g,J)$ belongs to the Gray--Hervella class $\mathcal{W}_{2}\oplus \mathcal{W}_{3}\oplus \mathcal{W}_{4}$ with $2s_{\rm C}=s$ then $(g,J)$ is K\"ahler. Recently Lejmi--Upmeier propose the following question \cite[Remark~3.3]{integra}:
\begin{Question} \label{question-princ}
Do higher-dimensional closed almost Hermitian non-K\"ahler manifolds with $2s_{\rm C}=s$ exist?
\end{Question}

In this paper we will investigate the Question~\ref{question-princ} in a class of high dimensional homogeneous spaces called {\it generalized flag manifolds}, equipped with an invariant almost Hermitian structure. It is well know that this class of homogeneous spaces have a very rich almost Hermitian geometry, see for instance \cite{rita, sn}. We will proceed by computing explicitly the Hermitian scalar curvatures, using the tools of Lie theory. The approach to compute explicitly the Hermitian scalar curvatures on flag manifolds is by analyzing the decomposition of the covariant derivative of the K\"ahler form with respect to the Levi-Civita connection, the codifferential of the Lee form, and the scalar Riemannian curvature. See Section~\ref{scal-hermit} for details.

Let us denote by $\nabla^t$ be the 1-parameter family of connection introduced by Gauduchon in~\cite{gau}:
\begin{align*}
g\big(\nabla^{t}_{X}Z,Y\big)={}&g(D_{X}Z,Y)-\frac{1}{2}g(J(D_{X}J)Z,Y)
\\
&+\frac{t}{4}g((D_{JY}J+JD_{Y}J)X,Z)-\frac{t}{4}g((D_{JZ}J+JD_{Z}J)X,Y),
\end{align*}
where $D$ is the Levi-Civita connection.
 It is well know that $\nabla^1$ coincides with the Chern connection.
 According to \cite{fu}, let us define the Hermitian scalar curvatures $s_{1}(t)=R^{t}(u_{\overline{i}},u_{i},u_{j},u_{\overline{j}})$
and $s_{2}(t)=R^{t}(u_{\overline{i}},u_{j},u_{i},u_{\overline{j}})$,
where $R^{t}$ is the curvature tensor associated to $\nabla^{t}$ and $\{u_{i}\}_{i=1,2,\dots,n}$ is a unitary frame. We will call $s_1(t)$ and $s_2(t)$ the first and second Hermitian scalar curvature. One of the main features of the generalized flag manifolds is that the first scalar curvature $s_1(t)$ is constant in $t$ and we will denote this curvature just by $s_1$ and it coincides with the Chern scalar curvature $s_{\rm C}$. Section \ref{scal-hermit} for details about the comments above.

With the discussion above in mind, let us rephrase the Question \ref{question-princ} in our context:
\begin{Question} 
Do higher-dimensional generalized flag manifolds $G/K$ equipped with $G$-inva\-ri\-ant almost Hermitian structure with $2s_{1}=s$ exist?
\end{Question}

We obtain the following results (for a description of invariant metrics and almost complex structures on flag manifolds, see Section \ref{section-prelim}):

\begin{Theorem}
Consider the flag manifold ${\rm SU}(3)/T^2$, equipped with an invariant almost Hermitian structure $(g,J)$, with invariant metric $g$ parametrized by a triple of positive numbers $(x,y,z)$ and $J$ non-integrable. Denote by $s_1$ the first Hermitian scalar curvature and by $s$ the Riemannian scalar curvature. The pair satisfies the equation $2s_1-s=0$ if, and only if,
$(g,J) \in\mathcal{W}_{1}\oplus \mathcal{W}_{2}$, and the metric $g$ satisfies one of the following relations:
	\begin{enumerate}\itemsep=0pt
	\item[$1.$] $z=3 (x+y)-2 \sqrt{2} \sqrt{x^2+3 x y+y^2}$, for $y> \big(2\sqrt{2}+3\big)x $ and $x>0$; or
$0<y< \big(3 -2 \sqrt{2}\big)x$ and $x>0$.

	\item[$2.$] $z=2 \sqrt{2} \sqrt{x^2+3 x y+y^2}+3 (x+y)$, for $x>0$ and $y>0$.
	\end{enumerate}
\end{Theorem}

\begin{Remark}
The solutions of the equation $2s_1-s=0$ for the {\em integrable} invariant almost complex structure $J$ on ${\rm SU}(3)/T^2$ are already know: $(g,J)\in\mathcal{W}_{3}$, being $J$ integrable and the metric $g$ satisfies $z=x+y$. In this case, the pair $(g,J)$ is K\"ahler, see \cite{fu2018twistor}.
\end{Remark}
\begin{Remark}
It is worth to point out that in the case of ${\rm SU}(3)/T^2$ if $J$ is non-integrable and the metric $g$ is parametrized by $(x,x,x)$ then $(g,J)\in \mathcal{W}_1$ (nearly-K\"ahler) and there is no solution to the equation $2s_1-s=0$.
\end{Remark}

\begin{Theorem}
Consider the complex projective plane $\mathbb{CP}^{3}={\rm Sp}(2)/{\rm Sp}(1)\times {\rm U}(1)$, equipped with an invariant almost Hermitian structure $(g,J)$, with invariant metric $g$ parametrized by a~pair of positive numbers $(x,y)$. Denote by $s_1$ the first Hermitian scalar curvature and by $s$ the Riemannian scalar curvature. The pair $(g,J)$ satisfies the equation $2s_1-s=0$ if, and only if,
\begin{enumerate}\itemsep=0pt
\item[$1.$] $(g,J) \in\mathcal{W}_{1}\oplus \mathcal{W}_{2}$, being $J$ non-integrable and the metric $g$ satisfies $y=2x\big(\sqrt{10} +3\big)$.
\item[$2.$] $(g,J)\in\mathcal{W}_{3}$, being $J$ integrable and the metric $g$ satisfies $y=2x$. In this case, the pair $(g,J)$ is K\"ahler.
\end{enumerate}
\end{Theorem}
\begin{Remark}
It is worthwhile to point out that in the case of ${\rm Sp}(2)/{\rm Sp}(1)\times {\rm U}(1)$ if $J$ is non-integrable and the metric $g$ is parametrized by $(x,x)$ then $(g,J)\in \mathcal{W}_1$ (nearly-K\"ahler) and there is no solution to the equation $2s_1-s=0$.
\end{Remark}

\begin{Theorem}
Consider the $10$-dimensional flag manifold $G_2/{\rm U}(2)$, where ${\rm U}(2)$ is represented by the short root of $G_2$, equipped with an invariant almost Hermitian structure $(g,J)$, with invariant metric $g$ parametrized by a pair of positive numbers $(x,y)$. Denote by $s_1$ the first Hermitian scalar curvature and by $s$ the Riemannian scalar curvature. The pair $(g,J)$ satisfies the equation $2s_1-s=0$ if, and only if,
\begin{enumerate}\itemsep=0pt
\item[$1.$] $(g,J) \in\mathcal{W}_{1}\oplus \mathcal{W}_{2}$, being $J$ non-integrable and the metric $g$ satisfies $y=2x\big(\sqrt{10} +3\big)$.
\item[$2.$] $(g,J)\in\mathcal{W}_{3}$, being $J$ integrable and the metric $g$ satisfies $y=2x$. In this case, the pair $(g,J)$ is K\"ahler.
\end{enumerate}
\end{Theorem}
\begin{Remark}
It is worth to point out that in the case of $G_2/{\rm U}(2)$, where ${\rm U}(2)$ is represented by the short root of $G_2$, if $J$ is non-integrable and the metric $g$ is parametrized by $(x,x)$ then $(g,J)\in \mathcal{W}_1$ (nearly-K\"ahler) and there is no solution to the equation $2s_1-s=0$.
\end{Remark}

Let $G$ be a compact simple Lie group and $T$ be a maximal torus in $G$. In the next two results we will consider the full flag manifolds $G/T$ as a total space of a homogeneous fibration over a~symmetric space. We will restrict ourselves to a family of invariant metrics on $G/T$ such that this fibration becomes a {\it Riemannian submersion with totally geodesic fibers}. These metrics appear naturally in the study of bifurcation theory for the Yamabe problem in homogeneous space, see for instance \cite{kenerson}. We will work out the Hermitian geometry and Hermitian scalar curvatures of the following homogeneous space: the (real) $12$-dimensional flag manifold ${\rm SU}(4)/T^3$ as total space of
\begin{gather*}
{\rm SU}(3)/T^2 \, \cdots {\rm SU}(4)/T^3 \, \mapsto \mathbb{CP}^3,
\end{gather*}
and the (real) $12$-dimensional flag manifold $G_2/T^2$ as total space
\begin{gather*}
S^2\times S^2 \, \cdots G_2/T^2 \, \mapsto G_2/{\rm SO}(4).
\end{gather*}

In both homogeneous spaces we will consider a 1-parameter family of invariant metrics $g_x$ obtained by re-scaling the {\it normal metric} in the direction of the {\it fibers} of the homogeneous fibration by the factor $x^2$.

\begin{Theorem} \label{teo-su4-t}
Let us consider the flag manifold ${\rm SU}(4)/T^3$, equipped with an invariant almost Hermitian structure $(g_x,J_i)$, $i=1,\ldots, 4$, where $g_x$ is 1-parameter family of invariant metric parametrized by the $6$-tuple $\big(x^{2},x^{2},1,x^{2},1,1\big)$ , $x>0$ and $J_i$, $i=1,\ldots, 4$ is an invariant almost complex structure. Denote by $s_1$ the first Hermitian scalar curvature and by $s$ the Riemannian scalar curvature. Then solution of the equation $2s_1-s=0$ is summarized in Table~$\ref{tab-teo-su4}$.

\begin{table}[h]\setlength{\tabcolsep}{4.5pt}
\renewcommand{\arraystretch}{1.2}
\centering
\caption{${\rm SU}(4)/T^3$.} \label{tab-teo-su4}
\begin{tabular}{c|c|c|c}\hline
 & GH Class & Condition & $2s_{1}-s=0$
 \\
\hline
$(J_{1},g_x)$ & $\mathcal{W}_{3}$ & $\forall x$ & $\nexists x$
\\
\hline
$(J_{2},g_x)$ & $\mathcal{W}_{1}\oplus \mathcal{W}_{3}$ & $\forall x$ & $x=\sqrt[4]{5}$
\\
\hline
$(J_{3},g_x)$ & $\mathcal{W}_{1}\oplus \mathcal{W}_{2}\oplus\mathcal{W}_{3}$ & $x\neq1$ & $\vphantom{\dfrac13^2}x=\sqrt{\frac{1}{3} \big(8-\sqrt{61}\big)}$ or $x=\sqrt{\frac{1}{3} \big(\sqrt{61}+8\big)}$
\\
 & $\mathcal{W}_{1}\oplus\mathcal{W}_{3}$ & $x=1$ & $\nexists x$
 \\
\hline
$(J_{4},g_x)$ & $\mathcal{W}_{1}\oplus \mathcal{W}_{2}\oplus\mathcal{W}_{3}$ & $x\neq1$ & $\vphantom{\dfrac13^2}x=\sqrt{\frac{1}{3} \big(\sqrt{165}+12\big)}$
\\
 & $\mathcal{W}_{1}\oplus\mathcal{W}_{3}$ & $x=1$ & $\nexists x$
 \\
\hline
\end{tabular}
\end{table}
\end{Theorem}

\begin{Theorem} \label{teo-g2-t2}
Let us consider the flag manifold $G_2/T^2$, equipped with an invariant almost Hermitian structure $(g_x,J_i)$, $i=1,\ldots, 32$, where $g_x$ is the $1$-parameter family of invariant metric parametrized by the $6$-tuple $\big(1,1,x^{2}, 1, x^{2},1 \big)$ and $J_i$, $i=1,\ldots, 32$ is an invariant almost complex structure. Denote by $s_1$ the first Hermitian scalar curvature and by $s$ the Riemannian scalar curvature. Then solution of the equation $2s_1-s=0$ is summarized in Table~$\ref{tab-teo-g2}$.

\begin{table}[h] \setlength{\tabcolsep}{4.5pt}
\renewcommand{\arraystretch}{1.2}
\centering
\caption{$G_2/T^2$.} \label{tab-teo-g2}
\begin{tabular}{c|c|c|c|c}
\hline
 & $J_i$ &GH Class & Condition & $2s_{1}-s=0$ \\
\hline
$(J_{i},g_x)$ & $i=1,2,3,10,21,30$ & $\mathcal{W}_{3}$ & $\forall x$ & $\nexists x$ \\
\hline
$(J_{i},g_x)$ & $i=4,\ldots,9,11,\ldots,20, $ & $\mathcal{W}_{1}\oplus \mathcal{W}_{2}\oplus\mathcal{W}_{3}$ & $x\neq1$ & at least one solution \\
 & $22,\ldots,29,31,32. $ & $\mathcal{W}_{1}\oplus\mathcal{W}_{3}$ & $x=1$ & $\nexists x$ \\
\hline
\end{tabular}
\end{table}
\end{Theorem}

\begin{Remark}
The explicit solutions for the equation $2s_{1}-s=0$ for the invariant Hermitian structures $(J_i,g_x)$, where $i=4,\ldots,9,11,\ldots,20, 22,\ldots,29,31,32$ and $x\neq 1$, are described in~Section~\ref{g2-full}.
\end{Remark}

\section[Preliminaries: generalized flag manifolds and invariant almost Hermitian structures] {Preliminaries: generalized flag manifolds \\and invariant almost Hermitian structures} \label{section-prelim}

In this section we recall some well know results about the geometry of generalized flag manifolds, from a Lie theoretical point of view. References about Riemannian and Hermitian geometry of flag manifolds are \cite{alek-flag,grego, besse,rita, sn}.
\subsection{Generalized flag manifolds}
Let $\mathfrak{g}^\mathbb{C}$ be a complex semi-simple Lie algebra. Given a Cartan subalgebra $\mathfrak{h}$ of $\mathfrak{g}^\mathbb{C}$, denotes by $\Pi$ the set of roots with respect to the pair $\big(\mathfrak{g}^\mathbb{C},\mathfrak{h}\big)$. We have the following decomposition
\begin{gather*}
\mathfrak{g}^\mathbb{C}=\mathfrak{h}\oplus\sum_{\alpha\in \Pi}\mathfrak{g}_{\alpha},
\end{gather*}
where $\mathfrak{g}_{\alpha}=\{X\in\mathfrak{g}\colon \forall H\in \mathfrak{h},[H,X]=\alpha(H)X \}$ denote the corresponding $1$-dimensional (complex) root space.

The Cartan--Killing form is defined by
\begin{gather*}
\langle X,Y\rangle=\tr(\ad(X)\ad(Y))
\end{gather*}
and its restriction to $\mathfrak{h}$ is non-degenerate. Given a root $\alpha\in \mathfrak{h}^{*}$ we define $H_{\alpha}$ by $\alpha(\cdot)=\langle H_{\alpha},\cdot\rangle$. Moreover we denote $\mathfrak{h}_{\mathbb{R}}=\Span_{\mathbb{R}}\{H_{\alpha}\colon \alpha\in \Pi\}$ and $\mathfrak{h}_{\mathbb{R}}^{*}$ being the real subspace of $\mathfrak{h}^{*}$ spanned by the roots.

Let us fix a Weyl basis of $\mathfrak{g}^\mathbb{C}$ given by
\begin{gather*}
X_{\alpha}\in\mathfrak{g}_{\alpha}, \qquad \text{such that} \quad
\langle X_{\alpha},X_{-\alpha}\rangle=1 \quad \text{and} \quad [X_{\alpha},X_{\beta}]=m_{\alpha,\beta}X_{\alpha+\beta},
\end{gather*}
with $m_{\alpha,\beta}\in\mathbb{R}$, $m_{-\alpha,-\beta}=-m_{\alpha,\beta}$ and $m_{\alpha,\beta}=0$ if $\alpha+\beta$ is not a root.

Let $\Pi^{+}\subset \Pi$ be a choice of positive roots, $\Sigma$ be the corresponding system of simple roots and~$\Theta$ be a subset of $\Sigma$. Let us fix the following notation: $\langle\Theta\rangle$ is the set of roots spanned by $\Theta$, $\Pi_{M}=\Pi\setminus\langle{\Theta}\rangle$ be the set of complementary roots and $\Pi_{M}^{+}$ be the set of complementary positive roots.

Let
\begin{gather*}
\mathfrak{p}_{\Theta}=\mathfrak{h}\oplus
\sum_{\alpha\in\langle\Theta\rangle^{+}}\mathfrak{g}_{\alpha}\oplus
\sum_{\alpha\in\langle\Theta\rangle^{+}}\mathfrak{g}_{-\alpha}\oplus
\sum_{\beta\in \Pi_{M}^{+}}\mathfrak{g}_{\beta}
\end{gather*}
be a parabolic sub-algebra of $\mathfrak{g}^\mathbb{C}$ determined by $\Theta$.

The generalized flag manifold $\mathbb{F}_{\Theta}$ is the homogeneous space
\begin{gather*}
\mathbb{F}_{\Theta}=G^\mathbb{C}/P_{\Theta},
\end{gather*}
where $G^\mathbb{C}$ is a complex connected Lie group with Lie algebra $\mathfrak{g}^\mathbb{C}$ and $P_{\Theta}$ is the normalizer of $\mathfrak{p}_{\Theta}$ in $G^\mathbb{C}$.

Let $\mathfrak{g}$ be the compact real form of $\mathfrak{g}^\mathbb{C}$. 
We have
\begin{gather*}
\mathfrak{g}=\Span_{\mathbb{R}}\{i\mathfrak{h}_{\mathbb{R}},A_{\alpha},iS_{\alpha}; \alpha\in \Pi\},
\end{gather*}
where $A_{\alpha}=X_{\alpha}-X_{-\alpha}$ and $S_{\alpha}=X_{\alpha}+X_{-\alpha}$. We remark that the Lie algebra $\mathfrak{g}$ is semi-simple.

Denote by $G$ the compact real form of $G^\mathbb{C}$ with $\Lie(G)=\mathfrak{g}$ and let $\mathfrak{k}_{\Theta}$ the Lie algebra of~$K_{\Theta}:=P_{\Theta}\cap G$. It is well known that the Lie group $K_{\Theta}\subset G$ is a centralizer of a torus.

We have
\begin{gather*}
\mathfrak{k}_{\Theta}^{\mathbb{C}}=\mathfrak{h}\oplus
\sum_{\alpha\in\langle\Theta\rangle^{+}}\mathfrak{g}_{\alpha}\oplus
\sum_{\alpha\in\langle\Theta\rangle^{+}}\mathfrak{g}_{-\alpha},
\end{gather*}
where $\mathfrak{k}_{\Theta}^{\mathbb{C}}$ denotes the complexification of the real Lie algebra $\mathfrak{k}_{\Theta}=\mathfrak{g}\cap \mathfrak{p}_{\Theta}$.

The Lie group $G$ also acts transitively on $\mathbb{F}_{\Theta}$ and we have
\begin{gather*}
\mathbb{F}_{\Theta}=G^\mathbb{C}/P_{\Theta}=G/(P_{\Theta}\cap G)=G/K_{\Theta}.
\end{gather*}

When $\Theta=\varnothing$ we have the following decomposition
\begin{gather*}
\mathfrak{g}^\mathbb{C}=\mathfrak{h}\oplus
\sum_{\beta\in \Pi^{+}}\mathfrak{g}_{\beta}\oplus
\sum_{\beta\in \Pi^{+}}\mathfrak{g}_{-\beta},
\end{gather*}
and therefore
\begin{gather*}
\mathfrak{p}_{\Theta}=\mathfrak{p}=\mathfrak{h}\oplus
\sum_{\beta\in \Pi^{+}}\mathfrak{g}_{\beta}
\end{gather*}
{\samepage
is a minimal parabolic sub-algebra (Borel sub-algebra) of $\mathfrak{g}^\mathbb{C}$ and
\begin{gather*}
\mathbb{F}=G^\mathbb{C}/P=G/T
\end{gather*}
is called {\it full flag manifold}, where $T=P\cap G$ is a maximal torus of $G$.

}

Recall that the flag manifold $\mathbb{F}_{\Theta}=G/K_{\Theta}$ is a reductive homogeneous space, that is, there exists a subspace $\mathfrak{m}$ of $\mathfrak{g}$ such that
\begin{gather*}
\mathfrak{g}=\mathfrak{k}_{\Theta}\oplus \mathfrak{m} \qquad \text{and} \qquad
\Ad(k)\mathfrak{m}\subseteq \mathfrak{m}\qquad \forall k\in K_{\Theta},
\end{gather*}
and one can identify the tangent space $T_{x_{0}}\mathbb{F}_{\Theta}$ with $\mathfrak{m}$, where $x_{0}=eK_{\Theta}$ is the origin of the $\mathbb{F}_{\Theta}$ (trivial coset).

Let us give a description of the tangent space $\mathfrak{m}$ in terms of the Lie algebra structure of $\mathfrak{g}$ as follows:
\begin{gather*}
\mathfrak{g}=\mathfrak{k}_{\Theta}\oplus\sum_{\beta\in \Pi_{M}}{\mathfrak{u}_{\beta}},
\end{gather*}
with $\mathfrak{u}_{\beta}=\mathfrak{g}\cap(\mathfrak{g}_{\beta}\oplus \mathfrak{g}_{-\beta})$ and for each root $\beta \in \Pi_{M}$, $\mathfrak{u}_{\beta}$ has real dimension two and it is spanned by $A_{\beta}$ and $\sqrt{-1}S_{\beta}$, and we have the following identification
\begin{gather*}
\mathfrak{m}=\sum_{\beta\in \Pi_{M}}\mathfrak{u}_{\beta}.
\end{gather*}

An important ingredient to study invariant tensors on homogeneous space is the {\em isotropy representation}. We will restrict ourselves to the situation of flag manifolds. In this case, since the flag manifolds are {\em reductive} homogeneous spaces it is well know that the isotropy representation is equivalent to the following representation
\begin{gather*}
\Ad(k)\mid_{\mathfrak{m}}\colon\ \mathfrak{m}\longrightarrow\mathfrak{m}.
\end{gather*}

The isotropy representation decomposes $\mathfrak{m}$ into irreducible components, that is,
\begin{gather*}
\mathfrak{m}=\mathfrak{m}_{1}\oplus\mathfrak{m}_{2}\oplus\cdots\oplus\mathfrak{m}_{n},
\end{gather*}
where each component $\mathfrak{m}_{i}$ satisfies $\Ad(K_{\Theta})(\mathfrak{m}_{i})\subset \mathfrak{m}_{i}$.
We have also that each component $\mathfrak{m}_{i}$ is irreducible, that is, the only invariant sub-spaces of $\mathfrak{m}_{i}$ by $\Ad(K_{\Theta})|_{\mathfrak{m}_{i}}$ are the trivial sub-spaces. We will call the sub-spaces $\mathfrak{m}_{i}$ by {\it isotropic summands} of the isotropy representation.

\begin{Remark}
In the sequel we will omit the symbol $\Theta$ whenever there is no risk of confusion. We will denote a flag manifold just by $\mathbb{F}=G/K$.
\end{Remark}

\subsection{Invariant metrics}
Let us denote by $(\cdot,\cdot)$ the Cartan--Killing form of $\mathfrak{g}$. For each $\Ad(K)-$invariant inner product $(\cdot ,\cdot)_{\Lambda}$ on $\mathfrak{m}$, there exists a unique $(\cdot,\cdot)-$self-adjoint, positive operator $\Lambda \colon \mathfrak{m}\longrightarrow\mathfrak{m}$ commuting with $\Ad(k)\left|_{\mathfrak{m}}\right.$ for all $k\in K$ such that
\begin{gather*}
(X,Y)_{\Lambda}=(\Lambda X,Y), \qquad X,Y\in\mathfrak{m}.
\end{gather*}

Therefore an invariant Riemannian metric $g$ on $\mathbb{F}$ is completely determined by the invariant inner product $(X,Y)_{\Lambda}$, and the inner product is determined by $\Lambda$.

The vectors $A_{\alpha}$, $\sqrt{-1}S_{\alpha}$, $\alpha\in\Pi$, are eigenvectors of $\Lambda$ associated to the same eigenvalue $\lambda_\alpha$.

The invariant inner product $(X,Y)_{\Lambda}$ admits a natural extension to a symmetric bilinear form on $\mathfrak{m}^{\mathbb{C}}$. On the complexified tangent space we have $\Lambda(X_{\alpha})=\lambda_{\alpha}X_{\alpha}$ with $\lambda_{\alpha}>0$ and $\lambda_{-\alpha}=\lambda_{\alpha}$.

\begin{Notation}
In the sequence of this work, we will abuse the notation and will denote the invariant metric $g$ just by the operator $\Lambda$ associated to the invariant inner product. We also denote the invariant metric $g$ just by a $n$-tuple of positive number $(\lambda_1,\ldots, \lambda_n)$ representing the eigenvalues of the operator $\Lambda$ and parametrized by the number of irreducible components.
\end{Notation}

\subsection{Invariant almost complex structures}
\begin{Definition}
An almost complex structure on the flag manifold $\mathbb{F}$ is a tensor $J$ such that for every point $x\in \mathbb{F}$, there is an endomorphism $J\colon T_{x}\mathbb{F}\longrightarrow T_{x}\mathbb{F}$ such that $J^{2}=-\id$.
\end{Definition}

\begin{Definition}
A $G$-invariant almost complex structure $J$ on $\mathbb{F}=G/K$ is an almost complex structure that satisfies
\begin{gather*}
J_{ux}={\rm d}E_{u}J_{x}{\rm d}E_{u^{-1}},\qquad \text{for all} \quad u\in G,
\end{gather*}
where ${\rm d}E_{u}\colon T(G/K) \to T(G/K)$ denotes the differential of the left translation by $u$, that is, for all $X\in T_{x}(G/K)$ we have
\begin{gather*}
{\rm d}E_{u}J_{x}X=J_{ux}{\rm d}E_{u}X.
\end{gather*}
\end{Definition}

The following result allows us to describe invariant almost complex structures on flag mani\-folds in terms of complex structure in a simple vector space, namely the tangent space at the origin (trivial coset) of the homogeneous space.

\begin{Proposition} \label{prop-corresp-iacs}
There exist an $1-1$ correspondence between a $G$-invariant almost complex structure $J$ and a linear endomorphism $J_{x_{0}}\colon  T_{x_{0}}\mathbb{F} \to T_{x_{0}}\mathbb{F}$ satisfying $J_{x_{0}}^2=-\id$ and commute with the isotropy representation, that is,
\begin{gather*}
\Ad^{G/K}(k)J_{x_{0}}=J_{x_{0}}\Ad^{G/K}(k) \qquad \text{for all} \quad k\in K.
\end{gather*}
\end{Proposition}

An interesting consequence of the Proposition \ref{prop-corresp-iacs} is that $J(\mathfrak{g}_{\alpha})=\mathfrak{g}_{\alpha}$, where $\mathfrak{g}_\alpha$ is the root space associated to the root $\alpha \in \Pi$. The eigenvalue of $J$ are $\pm \sqrt{-1}$ and the eigenvectors on $\mathfrak{m}^{\mathbb{C}}$ are $X_{\alpha}$, $\alpha\in\Pi$. Therefore $J(X_{\alpha})= \varepsilon_{\alpha} \sqrt{-1} X_{\alpha}$, with $\varepsilon_{\alpha}=\pm 1$ and $\varepsilon_{\alpha}=-\varepsilon_{-\alpha}$.

As a consequence of the discussion above we conclude that an invariant almost complex structure on $\mathbb{F}$ is completely described by a set of signals
\begin{gather*}
\left\{\varepsilon_{\alpha}=\pm 1, \ \alpha\in \Pi_M, \text{ satisfying } \varepsilon_{\alpha}=-\varepsilon_{-\alpha}\right\}.
\end{gather*}

\begin{Proposition}[{\cite[Proposition~13.4]{opa}}]
Consider the almost complex homogeneous space \mbox{$M=G/K$} and assume that the isotropy representation admits a decomposition into irreducible and pairwise non-equivalent components, namely, $\mathfrak{m}=\mathfrak{m}_{1}\oplus\mathfrak{m}_{2}\oplus\cdots\oplus\mathfrak{m}_{s}$. Then $M$ admits $2^{s}$~invariant almost complex structures.

If we identify the conjugated invariant almost complex structures, then $M$ admits $2^{s-1}$ invariant almost complex structures, up to conjugation.
\end{Proposition}

\section{Scalar curvatures of invariant almost Hermitian structures }\label{scal-hermit}

\subsection{Review: general results about curvatures of almost Hermitian structures}
Let $(M,J,g)$ be an almost Hermitian manifold with real dimension $2n$, with $J$ being an almost complex structure orthogonal with respect to the Riemannian metric $g$. A linear connection~$\nabla$ on~$M$ is Hermitian if it preserves the metric $g$ and the almost complex structure $J$, that is, $\nabla g=0$ and $\nabla J=0$ (we are not assuming that $J$ is integrable).

Lets us recall the $1$-parameter family of Hermitian connection defined by Gauduchon in \cite{gau} as follow:{\samepage
\begin{align*}
g\big(\nabla^{t}_{X}Z,Y\big)={}&g(D_{X}Z,Y)-\frac{1}{2}g(J(D_{X}J)Z,Y)
\\
&+\frac{t}{4}g((D_{JY}J+JD_{Y}J)X,Z)-\frac{t}{4}g((D_{JZ}J+JD_{Z}J)X,Y),
\end{align*}
where $D$ is the Levi-Civita connection.}

There are three special cases:
\begin{itemize}\itemsep=0pt\samepage
\item[$(i)$]$t=0$, $\nabla^{0}$ is the first canonical Hermitian connection, also know as Lichnerowicz connection or minimal connection.

\item[$(ii)$] $t=1$, $\nabla^{1}$ is the second canonical Hermitian connection, also know as Chern connection (this connection was used by Chern in the integrable case, see \cite{ch}).

\item[$(iii)$] $t=-1$, $\nabla^{-1}$ is the Bismut connection. In the integrable case, $\nabla^{-1}$ is characterized by its anti-symmetric torsion, see for instance \cite{ivanov,papa}.
\end{itemize}

According to \cite{fu}, let us define the Hermitian scalar curvatures $s_{1}(t)$ and $s_{2}(t)$ by
\begin{gather*}
s_{1}(t)=R^{t}(u_{\overline{i}},u_{i},u_{j},u_{\overline{j}}),
\qquad
s_{2}(t)=R^{t}(u_{\overline{i}},u_{j},u_{i},u_{\overline{j}}),
\end{gather*}
where $R^{t}$ is the curvature tensor associated to $\nabla^{t}$ and $\{u_{i}\}_{i=1,2,\dots,n}$ is a unitary frame. We call~$s_{1}(t)$ and $s_{2}(t)$ the {\em first} and {\em second} Hermitian scalar curvature, respectively.

For an Hermitian manifold equipped with the Chern connection $\nabla^{1}$ or the Bismut connection~$\nabla^{-1}$, the relations between the Hermitian scalar curvatures and Riemannian scalar curvatures were widely study, see for instance \cite{gau2,liu}.

Let $\{e_{1}, e_{2}, \dots, e_{2n}\}$ be a local orthonormal frame of $(M, g, J)$. Recall the $J$-twisted version of the Ricci tensor, called $J$-Ricci tensor and denoted by $\Ric_J$ (also called the $*$-Ricci tensor \cite{fu,sub,tri}) defined by
\begin{gather*}
\Ric_{J}(X, Y ) = R(e_{A}, X, J e_{A}, JY ).
\end{gather*}

The corresponding $J$-scalar curvature, denoted by $s_{J}$ is given by $s_{J} = \Ric_{J}(e_{A}, e_{A})$.

The Nijenhuis tensor $N$ is given by
\begin{gather*}
N(X,Y)=-[JX,JY]+J[JX,Y]+J[X,JY]+[X,Y],
\end{gather*}
where $X,Y\in \Gamma(TM)$.

Let us consider the fundamental (or K\"ahler) 2-form
\begin{gather*}
F(X,Y)=g(JX,Y).
\end{gather*}

The Lee form $\alpha_{F}$ of $(M,J,g)$ is defined by
\begin{gather*}
\alpha_{F}=J\delta F,
\end{gather*}
where $\delta=-{*{\rm d}*}$ is the codifferential with respect to $g$. The Lee form is also defined by
\begin{gather*}
{\rm d}F=({\rm d}F)_{0}+\dfrac{1}{n-1}\alpha_{F}\wedge F,
\end{gather*}
where $({\rm d}F)_{0}$ is the primitive part of ${\rm d}F$.

The covariant derivative of $F$ with respect to the Levi-Civita connection $D$ is
\begin{gather*}
(DF)(X,Y,Z)=\dfrac{1}{2}\left[{\rm d}F(X,Y,Z)-{\rm d}F(X,JY,JZ)-N(JX,Y,Z)\right].
\end{gather*}
Moreover,
\begin{gather*}
(DF)(X,Y,Z)=-(DF)(X,Z,Y)=-(DF)(X,JY,JZ).
\end{gather*}

{\samepage
The following expression of $DF$ will be very useful for our purposes (see \cite{fu,gau})
\begin{align*}
(DF)(X,Y,Z)={}&({\rm d}F)^{-}(X,Y,Z)-\dfrac{1}{2}N(JX,Y,Z)
\\
&+\dfrac{1}{2}[({\rm d}F)^{+}(X,Y,Z)-({\rm d}F)^{+}(X,JY,JZ)],
\end{align*}
where
$({\rm d}F)^{+}$ is the $(1,2)+(2,1)$-part of ${\rm d}F$ and $({\rm d}F)^{-}$ is the $(0,3)+(3,0)$-part of ${\rm d}F$.

}

Consider $N^{0}=N-\mathfrak{b}N$, where $\mathfrak{b}$ is the Bianchi projector, $\mathfrak{b}N^{0}=0$ and $\mathfrak{b}N$ is the anti-symmetric part of $N$ defined by
\begin{gather*}
\mathfrak{b}N(X,Y,Z)=\dfrac{1}{3}[N(X,Y,Z)+N(Y,Z,X)+N(Z,X,Y)].
\end{gather*}
According to \cite{gau}, we have
\begin{gather*}
3\mathfrak{b}N(X,Y,Z)=({\rm d}^{c}F)^{-}(X,Y,Z)=({\rm d}F)^{-}(JX,JY,JZ).
\end{gather*}

The four components described above $({\rm d}F)^{-}$, $N^{0}$, $({\rm d}F)^{+}_{0}$ and $\alpha_{F}$ provide us several geometric information about an almost Hermitian manifold. The $16$ classes of almost Hermitian structures described by Gray--Hervella in \cite{sub} correspond to the vanishing of the some subset of
\begin{gather*}
\big\{({\rm d}F)^{-}, N^{0}, ({\rm d}F)^{+}_{0},\alpha_{F}\big\}.
\end{gather*}

Let us describe some remarkable classes of almost Hermitian structures:
\begin{itemize}\itemsep=0pt

\item $\{0\}=$ {\em K\"ahler class}: all components vanish, $({\rm d}F)^{-}= N^{0}=({\rm d}F)^{+}_{0}=\alpha_{F}=0$.

\item $\mathcal{W}_{1}=$ {\em nearly-K\"ahler class}: $N^{0} = ({\rm d}F)^{+}_{0}= \alpha_{F}=0$.

\item $\mathcal{W}_{1}\oplus \mathcal{W}_{2}=$ {\em $(1,2)$-symplectic manifolds} (or {\it quas}i-K\"ahler): $({\rm d}F)^{+}_{0}=\alpha_{F}=0$.

\item $\mathcal{W}_{1}\oplus \mathcal{W}_{2}\oplus \mathcal{W}_{3}=$ {\em cosymplectic manifolds}: $\alpha_{F}=0$.
\end{itemize}

The Hermitian metric $g$ induces a natural inner product on $\wedge^{k}M$, the bundle of real $k$-forms, and also on $TM\otimes \wedge^{k}M$ the bundle of $TM$-valuated $k$-forms. The norm of the covariant derivative of the K\"ahler form is given by
\begin{align*}
\|DF\|^{2}=\|{\rm d}F\|^{2}+\dfrac{1}{4}\big\|N^{0}\big\|^{2}-\dfrac{2}{3}\|({\rm d}F)^{-}\|^{2}
=\|({\rm d}F)^{+}\|^{2}+\dfrac{1}{4}\big\|N^{0}\big\|^{2}+\dfrac{1}{3}\|({\rm d}F)^{-}\|^{2}.
\end{align*}
In particular, if $J$ is integrable, then $\|DF\|^{2}=\|{\rm d}F\|^{2}$.

\begin{Theorem}[{\cite[Theorem 4.3]{fu}}]\label{teo1}
Let $(M,g,J)$ be an almost Hermitian manifold of real dimension $2n$. Then
\begin{align*}
s_{1}(t)={}&\dfrac{s}{2}-\dfrac{5}{12}\|({\rm d}F)^{-}\|^{2}+\dfrac{1}{16}\big\|N^{0}\big\|^{2}+\dfrac{1}{4}\|({\rm d}F)_{0}^{+}\|^{2}
\\
&+\bigg[\dfrac{1}{4(n-1)}+\dfrac{t-1}{2}\bigg]\|\alpha_{F}\|^{2}+\dfrac{t-2}{2}\delta\alpha_{F},
\\
s_{2}(t)={}&\frac{s}{2}-\frac{1}{12}\|({\rm d}F)^{-}\|^{2}+\frac{1}{32}\big\|N^{0}\big\|^{2}-\dfrac{t^{2}-2t}{4}\|({\rm d}F)_{0}^{+}\|^{2}
\\
&-\bigg[\dfrac{t^{2}-2t}{4(n-1)}+\dfrac{(t+1)^{2}}{8}\bigg]\|\alpha_{F}\|^{2}-\dfrac{t+1}{2}\delta\alpha_{F},
\end{align*}
where $s$ denotes the Riemannian scalar curvature of $(M,g)$.
\end{Theorem}

Note that\vspace{-1ex}
\begin{align*}
2s_{1}(t)-s={}&-\dfrac{5}{6}\|({\rm d}F)^{-}\|^{2}+\dfrac{1}{8}\big\|N^{0}\big\|^{2}+\dfrac{1}{2}\|({\rm d}F)_{0}^{+}\|^{2}+\bigg[\dfrac{1}{2(n-1)}+(t-1)\bigg]\|\alpha_{F}\|^{2}
\\
&+(t-2)\delta\alpha_{F}.
\end{align*}

\begin{Proposition}[\cite{fu}]
The $J$-scalar curvature of $(M,g,J)$ is given by\vspace{-1ex}
\begin{gather*}
s_{J}=s-\dfrac{2}{3}\|({\rm d}F)^{-}\|^{2}+\dfrac{1}{4}\big\|N^{0}\big\|^{2}-\|\alpha_{F}\|^{2}-2\delta\alpha_{F},
\end{gather*}
where $s$ denotes the Riemannian scalar curvature of $(M,g)$.
\end{Proposition}

\subsection{Computations on generalized flag manifolds} \label{sec-computations}
Let us consider a flag manifold $G/K$. We will consider on $G/K$ an invariant metric $g$ and an invariant almost complex structure $J$. Recall the notation introduced in Section \ref{section-prelim}: we will represent an invariant metric $g$ by an $n$-tuple $(\lambda_1,\ldots, \lambda_n)$, where $n$ is the number of irreducible components of the isotropy representation of $\mathfrak{m}=T_o(G/K)$, that is, $\mathfrak{m}=\mathfrak{m}_1\oplus \dots \oplus \mathfrak{m}_n$. It~is worth to point out that if $\{ X_\alpha \}$ is a basis of $\mathfrak{m}$ induced by the Weyl basis then $X_\alpha$ is an eigenvector of the operator $\Lambda$ associated to the metric $g$, with same eigenvalue. This means that every vector in the irreducible component $\mathfrak{m}_i$ has length $\lambda_i$. The invariant almost complex structure $J$ is parametrized by a set of sign $\{ \varepsilon_\alpha \} =\pm 1 $, with $\varepsilon_{-\alpha}=-\varepsilon_\alpha$, where $\alpha$ is in the set of roots $\Pi$. Each vector $X_\alpha$ is an eigenvector of $J$ with eigenvalue $\varepsilon_\alpha\sqrt{-1}$ and every vector in an irreducible component $\mathfrak{m}_i$ is associated to the same $\varepsilon_i$.

We will write the Nijenhuis tensor in a similar way as \cite{sn}. We have $N=0$, except in the following situation:
\begin{gather*}
g(N(X_{\alpha},X_{\beta}),X_{\gamma}))=-\lambda_{\gamma}m_{\alpha,\beta}(\epsilon_{\alpha}\epsilon_{\beta}
+\epsilon_{\alpha}\epsilon_{\gamma}+\epsilon_{\beta}\epsilon_{\gamma}+1),
\end{gather*}
where $\alpha+\beta+\gamma=0$.

We also have
\begin{gather*}
g((N(X_{\alpha},X_{\beta}),J X_{\gamma})_=-\sqrt{-1}\lambda_{\gamma}m_{\alpha,\beta}(\epsilon_{\alpha}\epsilon_{\beta}\epsilon_{\gamma}
+\epsilon_{\alpha}+\epsilon_{\beta}+\epsilon_{\gamma}),
\end{gather*}
where $\alpha+\beta+\gamma=0$.

The next result was initially proved by San Martin-Negreiros \cite{sn} in the case of full flag manifolds and by R.~de Jesus in her Ph.D.~Thesis for the case of generalized flag manifold. We~include the proof here for the convenience of the reader.
\begin{Lemma}\label{lema-sm-neg}
Every invariant almost Hermitian structure $(g,J)$ on a generalized flag manifolds is cosymplectic.
\end{Lemma}
\begin{proof}
An almost Hermitian structure is cosymplectic if and only if
\begin{gather*}
\alpha_{F}(X)=\dfrac{1}{2n-1}\sum_{i}{\rm d}F(X,X_{i},Y_{i})=0,
\end{gather*}
where $\{X_{i}\}$ is a basis of the tangent space, $\{Y_{i}\}$ is a basis of the dual space with relation a~nondegenerate form $F$. We take $\{X_{i}\}=\big\{A_{\alpha}, \sqrt{-1}S_{\alpha};\alpha\in \Pi^{+}_M \big\}$ and $\{Y_{i}\}=\big\{\sqrt{-1}S_{\alpha}, A_{\alpha};\allowbreak \alpha\in \Pi^{+}_M \big\} $. So,
\begin{gather*}
{\rm d}F(X,X_{\alpha}-X_{-\alpha},{\rm i}X_{\alpha}+{\rm i}X_{-\alpha})
\\ \qquad
{}={\rm i}\,{\rm d}F(X,X_{\alpha},X_{\alpha})+{\rm i}\,{\rm d}F(X,X_{\alpha},X_{-\alpha})
-{\rm i}\,{\rm d}F(X,X_{-\alpha},X_{\alpha})-{\rm i}\,{\rm d}F(X,X_{-\alpha},X_{-\alpha})
\\ \qquad
{}=2{\rm i}({\rm d}F(X,X_{\alpha},X_{\alpha})+{\rm d}F(X,X_{\alpha},X_{-\alpha})).
\end{gather*}
We know that ${\rm d}F(X_{\alpha},X_{\beta}, X_{\gamma})=0$ unless $\alpha+\beta+\gamma=0$.
Thus, we take $X=X_{\beta}$ and
\begin{gather*}
\begin{cases}
{\rm d}F(X_{\beta},X_{\alpha}, X_{\alpha})=0,&\text{because if} \quad\alpha+\beta+\alpha=0 \quad \text{then} \quad \beta=-2\alpha,
\\
{\rm d}F(X_{\beta},X_{\alpha}, X_{-\alpha})=0,&\text{because if}\quad\alpha+\beta-\alpha=0 \quad \text{then} \quad \beta=0,
\end{cases}
\end{gather*}
and it is a contradiction.
Therefore, for every root $\gamma$,
\begin{gather*}
\alpha_{F}(X)=\dfrac{1}{2n-1}\sum_{\alpha>0}{\rm d}F(X_{\gamma},X_{\alpha}-X_{-\alpha},{\rm i}X_{\alpha}+{\rm i}X_{-\alpha})=0.\tag*{\qed}
\end{gather*}
\renewcommand{\qed}{}
\end{proof}

{\samepage\begin{Remark}\quad
\begin{enumerate}\itemsep=0pt
\item[1.] According to Lemma \ref{lema-sm-neg}, the Lee form vanishes identically for every generalized flag manifold. Therefore we have ${\rm d}F=({\rm d}F)_{0}$ and hence $({\rm d}F)^{+}=({\rm d}F)_{0}^{+}$. From now on, we will use $({\rm d}F)^{+}$ instead of $({\rm d}F)_{0}^{+}$.

\item[2.] We have $\delta \alpha_{F}=0$, where $\delta$ is the codifferential. In this case $s_{1}(t)$ is independent of $t$ (see Theorem \ref{teo1}), and we will denote $s_{1}(t)$ simply by $s_{1}$. In this case, $s_1$ coincides with the Chern scalar curvature $s_{\rm C}$.

\item[3.] In Theorem \ref{teo1} we will omit the terms which $\alpha_{F}$ appears.
\end{enumerate}
\end{Remark}

}

The next definition was introduced in \cite{sn} in order to obtain the classification of invariant almost Hermitian structures on full flag manifolds, and we will use this concept widely.

\begin{Definition}
Let $J$ be an invariant almost complex structure on the flag manifold $G/K$. The triple of roots $\alpha, \beta, \gamma\in \Pi^+_M$, with $\alpha+\beta+\gamma=0$ is said to be a $(0,3)$-triple if $\varepsilon_\alpha=\varepsilon_\beta=\varepsilon_\gamma$. It is a $(1,2)$-triple otherwise.
\end{Definition}

Let us consider an orthonormal basis on the tangent space at the origin with respect to the invariant metric, induced by the Weyl basis of the Lie algebra of $G$:
\begin{gather*}
\bigg\{E_{\alpha}=\dfrac{X_{\alpha}}{\sqrt{\lambda_{\alpha}}}\colon \alpha\in \Pi_{M}^{+}\bigg\}.
\end{gather*}

In the next lemmas we will compute the exterior derivatives and covariant derivatives of the K\"ahler form $F$ (with respect to the Levi-Civita connection) as well its $(p,q)$-parts, by using the Lie theoretical elements of the description of flag manifolds. Since the proofs are straightforward computation we will omit them.

\begin{Lemma}
The expression $(\epsilon_{\alpha}\epsilon_{\beta}+\epsilon_{\alpha}\epsilon_{\gamma}+\epsilon_{\beta}\epsilon_{\gamma}+1)$ is not zero if, and only if, $\{\alpha, \beta, \gamma\}$ is a~$(0,3)$-triple, that is, $\epsilon_{\alpha}=\epsilon_{\beta}=\epsilon_{\gamma}$.
\end{Lemma}
\begin{proof}\quad
\begin{enumerate}\itemsep=0pt
\item[$(i)$]
$\epsilon_{\alpha}=\epsilon_{\beta}\neq \epsilon_{\gamma}$:\quad
$\epsilon_{\alpha}\epsilon_{\beta}+\epsilon_{\alpha}\epsilon_{\gamma} +\epsilon_{\beta}\epsilon_{\gamma}+1=+1-1-1+1=0$,

\item[$(ii)$] $\epsilon_{\alpha}=\epsilon_{\gamma}\neq \epsilon_{\beta}$:\quad
$\epsilon_{\alpha}\epsilon_{\beta}+\epsilon_{\alpha}\epsilon_{\gamma} +\epsilon_{\beta}\epsilon_{\gamma}+1=-1+1-1+1=0$,

\item[$(iii)$] $\epsilon_{\beta}=\epsilon_{\gamma}\neq \epsilon_{\alpha}$:\quad
$\epsilon_{\alpha}\epsilon_{\beta}+\epsilon_{\alpha}\epsilon_{\gamma} +\epsilon_{\beta}\epsilon_{\gamma}+1=-1-1+1+1=0$,

\item[$(iv)$] $\epsilon_{\alpha}=\epsilon_{\beta}=\epsilon_{\gamma}$:\quad
$\epsilon_{\alpha}\epsilon_{\beta}+\epsilon_{\alpha}\epsilon_{\gamma} +\epsilon_{\beta}\epsilon_{\gamma}+1=+1+1+1+1=4$.\hfill \qed
\end{enumerate}\renewcommand{\qed}{}
\end{proof}

\begin{Lemma}
The covariant derivative of the K\"ahler form $F$ with respect to the Levi-Civita connection $D$ on a flag manifold $G/K$ is given by
\begin{gather*}
DF(E_{\alpha},E_{\beta},E_{\gamma})=
\begin{cases}
-\dfrac{\sqrt{-1}m_{\alpha,\beta}(\varepsilon_{\beta}+\varepsilon_{\gamma})(-\lambda_{\alpha} +\lambda_{\beta}+\lambda_{\gamma})}{2\sqrt{\lambda_{\alpha}\lambda_{\beta}\lambda_{\gamma}}}
&\text{if}\quad \alpha+\beta+\gamma = 0,
\\
0 &\text{otherwise}.
\end{cases}
\end{gather*}
\end{Lemma}

\begin{proof}
\begin{align*}
DF(E_{\alpha},E_{\beta},E_{\gamma})&=
\dfrac{g((D_{X_{\alpha}}J)X_{\beta},X_{\gamma})}{\sqrt{\lambda_{\alpha}\lambda_{\beta}\lambda_{\gamma}}}
\\
&=\dfrac{g(D_{X_{\alpha}}JX_{\beta},X_{\gamma})+g(D_{X_{\alpha}}X_{\beta},JX_{\gamma})}
{\sqrt{\lambda_{\alpha}\lambda_{\beta}\lambda_{\gamma}}}
\\
&=\dfrac{{\rm i}m_{\alpha,\gamma}(\lambda_{\gamma}-\lambda_{\alpha}+\lambda_{\beta})
(\varepsilon_{\beta}+\varepsilon_{\gamma})}{2\sqrt{\lambda_{\alpha}\lambda_{\beta}\lambda_{\gamma}}}
\\
&=-\dfrac{\sqrt{-1}m_{\alpha,\beta}(\varepsilon_{\beta}+\varepsilon_{\gamma})(-\lambda_{\alpha} +\lambda_{\beta}+\lambda_{\gamma})}{2\sqrt{\lambda_{\alpha}\lambda_{\beta}\lambda_{\gamma}}}.\tag*{\qed}
\end{align*}
\renewcommand{\qed}{}
\end{proof}

\begin{Lemma}
Let ${\rm d}F$ be the exterior derivative of the K\"ahler form $F$ on a flag manifold $G/K$. We have ${\rm d}F(E_{\alpha},E_{\beta},E_{\gamma})=0$, unless $\alpha+\beta+\gamma=0$. In this case we have
\begin{gather}
{\rm d}F(E_{\alpha},E_{\beta},E_{\gamma})=-\dfrac{\sqrt{-1}m_{\alpha,\beta} (\varepsilon_{\alpha}\lambda_{\alpha}+\varepsilon_{\beta}\lambda_{\beta} +\varepsilon_{\gamma}\lambda_{\gamma})}{\sqrt{\lambda_{\alpha}\lambda_{\beta}\lambda_{\gamma}}}.
\end{gather}
\end{Lemma}
\begin{proof}
The differential is given by
\begin{gather*}
{\rm d}F(E_{\alpha},E_{\beta},E_{\gamma})=\dfrac{g((D_{X_{\alpha}}J)X_{\beta},X_{\gamma}) -g((D_{X_{\beta}}J)X_{\alpha},X_{\gamma})+g((D_{X_{\gamma}}J)X_{\alpha},X_{\beta})} {\sqrt{\lambda_{\alpha}\lambda_{\beta}\lambda_{\gamma}}}.
\end{gather*}
Moreover
\begin{align*}
g((D_{X_{\alpha}}J)X_{\beta},X_{\gamma})&=g(D_{X_{\alpha}}JX_{\beta},X_{\gamma}) +g(D_{X_{\alpha}}X_{\beta},JX_{\gamma})
\\
&=\dfrac{\sqrt{-1}}{2}m_{\alpha,\gamma}(\lambda_{\gamma}-\lambda_{\alpha} +\lambda_{\beta})(\varepsilon_{\beta}+\varepsilon_{\gamma}).
\end{align*}
The other pieces of ${\rm d}F$ are obtained in a similar way and result follows.
\end{proof}

\begin{Lemma}
The $(0,2)$-component of the exterior derivative $DF$ is given by
\begin{gather*}
(DF)^{0,2}(E_{\alpha},E_{\beta},E_{\gamma})=
\begin{cases}
\dfrac{\sqrt{-1} m_{\alpha,\beta}(\lambda_{\alpha}\!-\!\lambda_{\beta}\!-\!\lambda_{\gamma})(\varepsilon_{\alpha} \!+\!\varepsilon_{\beta}\!+\!\varepsilon_{\gamma}\!+\!\varepsilon_{\alpha}\varepsilon_{\beta}\varepsilon_{\gamma})} {4\sqrt{\lambda_{\alpha}\lambda_{\beta}\lambda_{\gamma}}} &\text{if}\ \ \alpha\!+\!\beta\!+\!\gamma = 0,
\\
0 &\text{otherwise.}
\end{cases}
\end{gather*}
\end{Lemma}
\begin{proof}
By a direct computation we have
\begin{align*}
(DF)^{0,2}(E_{\alpha},E_{\beta},E_{\gamma})={}&\dfrac{g((D_{X_{\alpha}}J)X_{\beta},X_{\gamma}) +g(J(D_{JX_{\alpha}}J)X_{\beta},X_{\gamma})}{2\sqrt{\lambda_{\alpha}\lambda_{\beta}\lambda_{\gamma}}}
\\
={}&\dfrac{1}{2\sqrt{\lambda_{\alpha}\lambda_{\beta}\lambda_{\gamma}}} [g(D_{X_{\alpha}}JX_{\beta},X_{\gamma})+g(D_{X_{\alpha}}X_{\beta},J X_{\gamma})
\\
&\hphantom{\dfrac{1}{2\sqrt{\lambda_{\alpha}\lambda_{\beta}\lambda_{\gamma}}} [}+g(D_{J X_{\alpha}}X_{\beta},X_{\gamma})-g(D_{J X_{\alpha}}J{X_{\beta}},J X_{\gamma})]
\\
={}&\dfrac{\sqrt{-1}m_{\alpha,\beta}(\lambda_{\alpha}-\lambda_{\beta} -\lambda_{\gamma})(\varepsilon_{\alpha} +\varepsilon_{\beta}+\varepsilon_{\gamma} +\varepsilon_{\alpha}\varepsilon_{\beta}\varepsilon_{\gamma})} {4\sqrt{\lambda_{\alpha}\lambda_{\beta}\lambda_{\gamma}}},
\end{align*}
and the result follows.
\end{proof}

Since the $(2,0)$-component of $DF$ is defined by
\begin{gather*}
(DF)^{2,0}(E_{\alpha},E_{\beta},E_{\gamma})=\dfrac{g((D_{X_{\alpha}}J)X_{\beta},X_{\gamma})-g(J(D_{JX_{\alpha}}J)X_{\beta},X_{\gamma})}{2\sqrt{\lambda_{\alpha}\lambda_{\beta}\lambda_{\gamma}}}
\end{gather*}
we obtain the next result in similar way.

\begin{Lemma}
The $(2,0)$-component of the exterior derivative $DF$ is given by
\begin{gather*}
(DF)^{2,0}(E_{\alpha},E_{\beta},E_{\gamma})
\\ \qquad
{}=
\begin{cases}
-\dfrac{\sqrt{-1} m_{\alpha,\beta}(-\varepsilon_{\alpha}+\varepsilon_{\beta}+\varepsilon_{\gamma} -\varepsilon_{\alpha}\varepsilon_{\beta}\varepsilon_{\gamma})(-\lambda_{\alpha}+\lambda_{\gamma} +\lambda_{\beta})}{4\sqrt{\lambda_{\alpha}\lambda_{\beta}\lambda_{\gamma}}} & \text{if}\quad \alpha+\beta+\gamma = 0,
\\
0& \text{otherwise.}
\end{cases}
\end{gather*}
\end{Lemma}

\begin{Lemma}
The $(0,3)+(3,0)$-component of ${\rm d}F$, denoted by $({\rm d}F)^{-}$, is given by
\begin{gather*}({\rm d}F)^{-}(E_{\alpha},E_{\beta},E_{\gamma})
\\ \qquad
{}=
\begin{cases}
-\dfrac{\sqrt{-1} m_{\alpha,\beta}(\lambda_{\alpha}+\lambda_{\beta}+\lambda_{\gamma})(\varepsilon_{\alpha} +\varepsilon_{\beta}+\varepsilon_{\gamma}+\varepsilon_{\alpha}\varepsilon_{\beta}\varepsilon_{\gamma})} {4\sqrt{\lambda_{\alpha}\lambda_{\beta}\lambda_{\gamma}}} &\text{if}\quad \alpha+\beta+\gamma = 0,
\\
0 & {otherwise.}
\end{cases}
\end{gather*}
\end{Lemma}
\begin{proof}
According to \cite{gau}, the $(0,3)+(3,0)$-component of ${\rm d}F$ is defined by
\begin{align*}
({\rm d}F)^{-}(E_{\alpha},E_{\beta},E_{\gamma})&= \dfrac{\mathfrak{b}(DF)^{0,2}(X_{\alpha},X_{\beta},X_{\gamma})} {\sqrt{\lambda_{\alpha}\lambda_{\beta}\lambda_{\gamma}}}
\\
&=\dfrac{(DF)^{0,2}(X_{\alpha},X_{\beta},X_{\gamma})\!+\!(DF)^{0,2}(X_{\beta},X_{\gamma},X_{\alpha}) \!+\!(DF)^{0,2}(X_{\gamma},X_{\alpha},X_{\beta})}{\sqrt{\lambda_{\alpha}\lambda_{\beta}\lambda_{\gamma}}}
\\
&=-\dfrac{ \sqrt{-1} m_{\alpha,\beta}(\lambda_{\alpha}+\lambda_{\beta}+\lambda_{\gamma})(\varepsilon_{\alpha} +\varepsilon_{\beta}+\varepsilon_{\gamma}+\varepsilon_{\alpha}\varepsilon_{\beta}\varepsilon_{\gamma})} {4\sqrt{\lambda_{\alpha}\lambda_{\beta}\lambda_{\gamma}}},
\end{align*}
and the result follows.
\end{proof}

\begin{Lemma}
The $(1,2)+(2,1)$-component of ${\rm d}F$, denoted by $({\rm d}F)^{+}$, is given by
\begin{gather*}
({\rm d}F)^{+}(E_{\alpha},E_{\beta},E_{\gamma})
\\ \qquad
{}=\begin{cases}
-\dfrac{\sqrt{-1} m_{\alpha,\beta}\{4(\varepsilon_{\alpha}\lambda_{\alpha} \!+\varepsilon_{\beta}\lambda_{\beta}\!+\varepsilon_{\gamma}\lambda_{\gamma}) -(\varepsilon_{\alpha}\!+\varepsilon_{\beta}\!+\varepsilon_{\gamma} \!+\varepsilon_{\alpha}\varepsilon_{\beta}\varepsilon_{\gamma})(\lambda_{\alpha}\!+\lambda_{\beta} \!+\lambda_{\gamma})\}}{4\sqrt{\lambda_{\alpha}\lambda_{\beta}\lambda_{\gamma}}}
\\[-1ex]
\hspace{98mm}\text{if} \quad \alpha+\beta+\gamma = 0,
\\
0 \qquad \text{otherwise.}
\end{cases}
\end{gather*}
\end{Lemma}
\begin{proof}
 The $(1,2)+(2,1)$-component of ${\rm d}F$ is defined by
\begin{gather*}
({\rm d}F)^{+}(E_{\alpha},E_{\beta},E_{\gamma})
\\ \qquad
{}=\dfrac{3\mathfrak{b}(DF)^{2,0}(X_{\alpha},X_{\beta},X_{\gamma})} {\sqrt{\lambda_{\alpha}\lambda_{\beta}\lambda_{\gamma}}}
\\ \qquad
{}=\dfrac{(DF)^{2,0}(X_{\alpha},X_{\beta},X_{\gamma})+(DF)^{2,0}(X_{\beta},X_{\gamma},X_{\alpha}) +(DF)^{2,0}(X_{\gamma},X_{\alpha},X_{\beta})}{\sqrt{\lambda_{\alpha}\lambda_{\beta}\lambda_{\gamma}}}
\\ \qquad
{}=-\dfrac{\sqrt{-1} m_{\alpha,\beta}\{4(\varepsilon_{\alpha}\lambda_{\alpha} +\varepsilon_{\beta}\lambda_{\beta}+\varepsilon_{\gamma}\lambda_{\gamma}) -(\varepsilon_{\alpha}+\varepsilon_{\beta}+\varepsilon_{\gamma} +\varepsilon_{\alpha}\varepsilon_{\beta}\varepsilon_{\gamma}) (\lambda_{\alpha}+\lambda_{\beta}+\lambda_{\gamma})\}} {4\sqrt{\lambda_{\alpha}\lambda_{\beta}\lambda_{\gamma}}},
\end{gather*}
and the result follows.
\end{proof}

\begin{Remark}
It is worth to point out an alternative description of $DF^{0,2}$ and $DF^{2,0}$ that will be very useful in our work:
\begin{gather*}
DF^{0,2}(E_{\alpha},E_{\beta},E_{\gamma})=\dfrac{2({\rm d}F)^{-}(X_{\alpha},X_{\beta},X_{\gamma}) -N(JX_{\alpha},X_{\beta},X_{\gamma})}{2\sqrt{\lambda_{\alpha}\lambda_{\beta}\lambda_{\gamma}}},
\end{gather*}
and
\begin{gather*}
DF^{2,0}(E_{\alpha},E_{\beta},E_{\gamma})=\dfrac{({\rm d}F)^{+}(X_{\alpha},X_{\beta},X_{\gamma})-({\rm d}F)^{+}(X_{\alpha},JX_{\beta},JX_{\gamma})}{\sqrt{\lambda_{\alpha}\lambda_{\beta}\lambda_{\gamma}}}.
\end{gather*}
We also point out a description of $({\rm d}F)^-$ in terms of the Nijenhius tensor $N$ given by (see \cite{gau}):
\begin{gather*}
({\rm d}F)^{-}(X,Y,Z)=-N(JX,Y,Z)+N(JY,X,Z)-N(JZ,X,Y).
\end{gather*}
The description of $N^0$ in terms of the Nijenhius tensor $N$ is given in \cite{gau}):
\begin{gather*}
N^{0}(X,Y,Z)=\frac{2}{3}N(X,Y,Z)-\frac{1}{3}N(Y,Z,X)-\frac{1}{3}N(Z,X,Y).
\end{gather*}
\end{Remark}

\begin{Lemma}
We have
\begin{gather*}
N^{0}(E_{\alpha},E_{\beta},E_{\gamma})=\begin{cases}
-\dfrac{\sqrt{-1} \, m_{\alpha,\beta} (\epsilon_{\alpha}\epsilon_{\beta}\epsilon_{\gamma} \!+\epsilon_{\alpha}\!+\epsilon_{\beta}\!+\epsilon_{\gamma}) (\lambda_{\alpha}\!+\lambda_{\beta}\!+\lambda_{\gamma})}{4 \sqrt{\lambda_{\alpha}\lambda_{\beta}\lambda_{\gamma}}}  & \text{if}\quad \alpha\!+\beta\!+\gamma = 0,
\\
0 & \text{otherwise}.
\end{cases}
\end{gather*}
\end{Lemma}

\begin{Remark}
When $\alpha +\beta+\gamma=0$ is a $(1,2)$-triple we have
\begin{gather*}
({\rm d}F)^{-}(E_{\alpha},E_{\beta},E_{\gamma})=N^{0}(E_{\alpha},E_{\beta},E_{\gamma})=0
\end{gather*}
and
\begin{gather*}
({\rm d}F)^{+}(E_{\alpha},E_{\beta},E_{\gamma})=\dfrac{-{\rm i}m_{\alpha,\beta}(\varepsilon_{\alpha}\lambda_{\alpha}+\varepsilon_{\beta}\lambda_{\beta}+\varepsilon_{\gamma}\lambda_{\gamma})}{\sqrt{\lambda_{\alpha}\lambda_{\beta}\lambda_{\gamma}}}.
\end{gather*}
Otherwise when $\alpha +\beta+\gamma=0$ is a $(0,3)$-triple we have $({\rm d}F)^{+}(E_{\alpha},E_{\beta},E_{\gamma})=0$,
\begin{gather*}
({\rm d}F)^{-}(E_{\alpha},E_{\beta},E_{\gamma})
=-\dfrac{{\rm i}m_{\alpha,\beta}(\lambda_{\alpha}+\lambda_{\beta}+\lambda_{\gamma})}{\sqrt{\lambda_{\alpha}\lambda_{\beta}\lambda_{\gamma}}}
\end{gather*}
and
\begin{gather*}
N^{0}(E_{\alpha},E_{\beta},E_{\gamma})=\dfrac{4m_{\alpha,\beta}(-2\lambda_{\alpha}+\lambda_{\beta}+\lambda_{\gamma})}{3 \sqrt{\lambda_{\alpha}\lambda_{\beta}\lambda_{\gamma}}}.
\end{gather*}
\end{Remark}

\begin{Remark}Let us recall some useful relations using the derivatives of K\"ahler form. See \cite{gau} for further details.
\begin{itemize}\itemsep=0pt
\item The pair $(g,J)$ is integrable if, and only if, $D_{JX}J=JD_{X}J$.
Note that
\begin{gather*}
(g,J) \ \text{is integrable} \Leftrightarrow N\equiv 0 \Leftrightarrow DF^{0,2}=0\Leftrightarrow D_{JX}J=JD_{X}J.
\end{gather*}

\item The pair $(g,J)$ is $(1,2)$-symplectic if, and only if, $D_{JX}J=-JD_{X}J$.
We have
\begin{gather*}
(g,J) \ \text{is} \ \text{(1,2)-symplectic} \Leftrightarrow ({\rm d}F)^{+}\equiv 0 \Leftrightarrow DF^{2,0}=0\Leftrightarrow D_{JX}J=-JD_{X}J.
\end{gather*}

\item Note that
\begin{gather*}
\begin{cases}
DF=DF^{0,2}+DF^{2,0},
\\
{\rm d}F=({\rm d}F)^{-}+({\rm d}F)^{+}.
\end{cases}
\end{gather*}
\end{itemize}
\end{Remark}

In order to calculate the Hermitian scalar curvatures, the first step is computing the norms of $N^{0}$, $({\rm d}F)^{-}$, $({\rm d}F)^{+}$ and $\alpha_{F}$, as described in \cite{fu}. Since these computations are standard we omit the details.

\begin{Proposition} \label{prop-norma}
We have
\begin{gather*}
\big\|N^{0}\big\|^{2}=\sum_{\alpha+\beta+\gamma=0}(m_{\alpha,\beta})^{2} \dfrac{(\epsilon_{\alpha}\epsilon_{\beta}\epsilon_{\gamma}+\epsilon_{\alpha} +\epsilon_{\beta}+\epsilon_{\gamma})^2}{54\lambda_{\alpha}\lambda_{\beta}\lambda_{\gamma}}
\\ \hphantom{\big\|N^{0}\big\|^{2}=}
{}\times\big[(-2\lambda_{\alpha}+\lambda_{\beta}+\lambda_{\gamma})^{2}+ (-2\lambda_{\gamma}+\lambda_{\alpha}+\lambda_{\beta})^{2}+(-2\lambda_{\beta} +\lambda_{\alpha}+\lambda_{\gamma})^{2}\big],
\\
\|({\rm d}F)^{-}\|^{2}=\sum_{\alpha+\beta+\gamma=0}\dfrac{(m_{\alpha,\beta})^{2} (\epsilon_{\alpha}\epsilon_{\beta}\epsilon_{\gamma}+\epsilon_{\alpha}+\epsilon_{\beta}+\epsilon_{\gamma} )^2(\lambda_{\alpha}+\lambda_{\beta}+\lambda_{\gamma})^{2}} {96\lambda_{\alpha}\lambda_{\beta}\lambda_{\gamma}},
\\
\|DF\|^{2}=\sum_{\alpha+\beta+\gamma=0}\dfrac{m_{\alpha,\beta}^{2}}{3\lambda_{\alpha}\lambda_{\beta} \lambda_{\gamma}}\bigg[\dfrac{1}{4}(\varepsilon_{\beta}+\varepsilon_{\gamma})^{2} (-\lambda_{\alpha}+\lambda_{\beta}+\lambda_{\gamma})^{2}
+ \dfrac{1}{4}(\varepsilon_{\alpha}+\varepsilon_{\gamma})^{2}(\lambda_{\alpha}-\lambda_{\beta} +\lambda_{\gamma})^{2}
\\ \hphantom{\|DF\|^{2}=}
{}+\dfrac{1}{4}(\varepsilon_{\alpha}+\varepsilon_{\beta})^{2}(\lambda_{\alpha} +\lambda_{\beta}-\lambda_{\gamma})^{2}\bigg],
\\
\|({\rm d}F)^{+}\|^{2}=\!\sum_{\alpha\!+\beta\!+\gamma=0}\!\!\!\!\dfrac{m_{\alpha,\beta}^{2} \{4(\varepsilon_{\alpha}\lambda_{\alpha}\!+\varepsilon_{\beta}\lambda_{\beta} \!+\varepsilon_{\gamma}\lambda_{\gamma})\!-(\varepsilon_{\alpha}\!+\varepsilon_{\beta} \!+\varepsilon_{\gamma}\!+\varepsilon_{\alpha}\varepsilon_{\beta}\varepsilon_{\gamma}) (\lambda_{\alpha}\!+\lambda_{\beta}\!+\lambda_{\gamma})\}} {96\lambda_{\alpha}\lambda_{\beta}\lambda_{\gamma}}^{2}\!.
\end{gather*}
\end{Proposition}

\subsection[The flag manifold SU(3)/T\textasciicircum{}2]
{The flag manifold $\boldsymbol{{\rm SU}(3)/T^{2}}$}

Let us consider the 6-dimensional full flag manifold $\mathbb{F}(3)={\rm SU}(3)/T^2$. Recall the Cartan subalgebra $\mathfrak{h}$ of $\mathfrak{su}(3)$ is given by
\begin{gather*}
\mathfrak{h}=\{\Diag(x_{1},x_{2},x_{3})\colon x_{1}+x_{2}+x_{3}=0,\ x_{1},x_{2},x_{3}\in\mathbb{C}\}.
\end{gather*}

The set of positive roots is given by $\alpha_{12}=x_{1}-x_{2}$, $\alpha_{23}=x_{2}-x_{3}$ and $\alpha_{13}=x_{1}-x_{3}$. Recall that on full flag manifolds, the set of complementary roots $\Pi_{M}$ coincides with set of roots $\Pi$. The isotropy representation of $\mathbb{F}(3)$ admits three irreducible components, and each component corresponds to a unique positive root $\Pi^{+}$, that is,
\begin{align*}
\mathfrak{m}&=\mathfrak{m}_{1}\oplus\mathfrak{m}_{2}\oplus \mathfrak{m}_{3}
\\
&=\mathfrak{u}_{12}\oplus \mathfrak{u}_{23}\oplus \mathfrak{u}_{13},
\end{align*}
where $\mathfrak{u}_{ij}=\mathfrak{su}(3)\cap (\mathfrak{g}_{ij}\oplus \mathfrak{g}_{ji})$, where $\mathfrak{g}_{ji}$ is the root space associated to the root $\alpha_{ij}$.
Therefore there exist two {\it invariant} almost complex structures, up to conjugation and equivalence:
\begin{gather*}
J_{1}=(+,+,+)\qquad \text{and}\qquad J_{2}=(+,+,-),
\end{gather*}
where the almost complex structure $J_1$ is given by $ \varepsilon_{\alpha_{12}} = +1$, $\varepsilon_{\alpha_{21}} = -1$, $\varepsilon_{\alpha_{23}} = +1$, $ \varepsilon_{\alpha_{32}} = -1$, $\varepsilon_{\alpha_{13}} = +1$, $\varepsilon_{\alpha_{31}} = -1$, and so on. We remark that $J_1$ is integrable and $J_2$ is non-integrable, see for instance \cite{sn}.

We will denote an invariant metric $g$ by the triple of positive numbers $(x,y,z)$. To be consistent with Section \ref{sec-computations}, we set $\lambda_{12}=x$, $\lambda_{23}=y$ and $\lambda_{13}=z$. We will consider the orthonormal basis of $TM_{\mathbb{C}}$ given by
\begin{gather*}
\bigg\{\dfrac{X_{12}}{\sqrt{x}},\dfrac{X_{23}}{\sqrt{y}},\dfrac{X_{13}}{\sqrt{z}},\dfrac{X_{21}}{\sqrt{x}},
\dfrac{X_{32}}{\sqrt{y}},\dfrac{X_{31}}{\sqrt{z}}\bigg\},
\end{gather*}
where $X_{ij}\in \mathfrak{g}_{ij}$ represents the elements of Weyl basis of $\mathfrak{sl}(3,\mathbb{C})$.

\begin{Proposition}[\cite{grego}]
Let us consider the full flag manifold ${\rm SU}(3)/T$, with invariant metric $g$ parametrized by $(x,y,z)$. Then the Riemannian scalar curvature of $({\rm SU}(3)/T,g)$ is given by
\begin{gather*}
s= -\dfrac{x^{2}+y^{2}-6yz+z^{2}-6x(y+z)}{6xyz}.
\end{gather*}
\end{Proposition}

\begin{Lemma}
The invariant Hermitian structure $(g, J_{2})$ does not admit $(1,2)$-triple. The $(0,3)$-triples are given by $\alpha_{12}+\alpha_{23}+\alpha_{31}=0$ and $\alpha_{21}+\alpha_{32}+\alpha_{13}=0$.
\end{Lemma}

\begin{Proposition} \label{curv-su3}
Let us consider the flag manifold ${\rm SU}(3)/T^2$, equipped with invariant Hermitian structure $(g, J_{2})$. Denote by $s$ the Riemannian scalar curvature and $s_1$ the first Hermitian scalar curvature. Then the invariant metric $g$ that are solution of the equation $2s_1-s=0$ are given by:
 \begin{itemize}\itemsep=0pt
\item[$(i)$] $z=3 (x+y)-2 \sqrt{2} \sqrt{x^2+3 x y+y^2}$, for $y>2 \sqrt{2} \sqrt{x^2}+3 x$ and $x>0$; or
$0<y<3 x-2 \sqrt{2} \sqrt{x^2}$ and $x>0$,
\item[$(ii)$] $z=2 \sqrt{2} \sqrt{x^2+3 x y+y^2}+3 (x+y)$, for $x>0$ and $y>0$.
\end{itemize}
\end{Proposition}

\begin{proof}
By Proposition \ref{prop-norma} we have $\|({\rm d}F)^{+}\|^{2}=0$,
\begin{gather*}
\|({\rm d}F)^{-}\|^{2}=\dfrac{(x+y+z)^{2}}{3xyz},
\\
\big\|N^{0}\big\|^{2}=\dfrac{16\big[(-2x\!+\!y\!+\!z)^{2}\!+\!(x\!-\!2y\!+\!z)^{2}\!+\!(x\!+\!y\!-\!2z)^{2}\big]}{27 xyz}
=\dfrac{32 \big(x^2 \!+\! y^2 \!-\! y z \!+\! z^2 \!-\! x (y \!+\! z)\big)}{9 x y z},
\\
\|DF\|^{2}=\dfrac{3 x^2 + 3 y^2 - 2 y z + 3 z^2 - 2 x (y + z)}{3 x y z}.
\end{gather*}

The first Hermitian scalar curvature is
\begin{gather*}
s_{1}=\dfrac{s}{2}-\dfrac{5}{12}\|({\rm d}F)^{-}\|^{2}+\dfrac{1}{16}\big\|N^{0}\big\|^{2}+\dfrac{1}{4}\|({\rm d}F)^{+}\|^{2}=0.
\end{gather*}

Therefore in order to solve the equation $2s_{1}-s=0$ we need to find the zeros of the Riemannian scalar curvature.

We have $s=0$ if, and only if,
\begin{itemize}\itemsep=0pt
\item[$(i)$] $z=3 (x+y)-2 \sqrt{2} \sqrt{x^2+3 x y+y^2}$, for $y>2 \sqrt{2} \sqrt{x^2}+3 x$ and $x>0$; or
$0<y<3 x-2 \sqrt{2} \sqrt{x^2}$ and $x>0$,
\item[$(ii)$] $z=2 \sqrt{2} \sqrt{x^2+3 x y+y^2}+3 (x+y)$, for $x>0$ and $y>0$.
\hfill \qed
\end{itemize}\renewcommand{\qed}{}
\end{proof}

An immediate consequence of the proof of Proposition \ref{curv-su3}, by analyzing the vanishing of the components $({\rm d}F)^{-}$, $({\rm d}F)^{+}$, $N^0$ is the following result:
\begin{Proposition}\label{prop-c1}
With the notation above, the invariant almost Hermitian structure $(g, J_{2})$ on ${\rm SU}(3)/T^2$ belongs to the Gray--Hervella class $\mathcal{W}_{1}\oplus \mathcal{W}_{2}$. The pair $(g, J_{2})$ belongs to the class~$\mathcal{W}_{1}$ $($nearly K\"ahler$)$ if, and only if, the metric $g$ is parametrized by $x=y=z$.
\end{Proposition}

For the sake of completeness we will compute the scalar curvatures $s_2=s_{2}(t)$ and $s_J$. We~will use the ingredients computed in the Proposition \ref{curv-su3}.

\begin{Proposition}
Let us consider the flag manifold ${\rm SU}(3)/T^2$, equipped with invariant Hermitian structure $(g, J_{2})$. The second Hermitian scalar curvature $s_2$ and the $J$-scalar curvature are given by
\begin{gather*}
s_{2}=\frac{1}{3} \bigg(\frac{1}{x}+\frac{1}{y}+\frac{1}{z}\bigg),
\qquad
s_{J}=\frac{3 x^2-2 x (y+z)+3 y^2-2 y z+3 z^2}{6 x y z}.
\end{gather*}
\end{Proposition}

Let us consider now the invariant Hermitian structure $(g, J_{1})$.

\begin{Lemma}
The invariant Hermitian structure $(g, J_{1})$ does not admit $(0,3)$-triple. The $(1,2)$-triples are given by $\alpha_{12}+\alpha_{23}+\alpha_{31}=0$ and $\alpha_{21}+\alpha_{32}+\alpha_{13}=0$.
\end{Lemma}

\begin{Proposition}\label{prop-su3-j1}
Let us consider the flag manifold ${\rm SU}(3)/T^2$, equipped with invariant Hermitian structure $(g, J_{1})$. Denote by $s$ the Riemannian scalar curvature and $s_1$ the first Hermitian scalar curvature. Then the invariant metric $g$ that is solution of the equation $2s_1-s=0$ is given by $z=x+y$. In this case, the pair $(g, J_{1})$ is a K\"ahler structure.
\end{Proposition}
\begin{proof}
We have $\|({\rm d}F)^{-}\|^{2}=\|N^0\|^{2}=0$ since the $(1,2)$-triples are {\it exactly} the zero-sum triples of roots.

By Proposition \ref{prop-norma} we have
\begin{gather*}
\|({\rm d}F)^{+}\|^{2}=\dfrac{(\varepsilon_{12}\lambda_{12}+\varepsilon_{23}\lambda_{23} +\varepsilon_{31}\lambda_{13})^{2}}{3\lambda_{12}\lambda_{23}\lambda_{13}} =\dfrac{(\lambda_{12}+\lambda_{23}-\lambda_{13})^{2}}{3xyz}
=\dfrac{(x+y-z)^{2}}{3xyz},
\\
\|DF\|^2=\dfrac{(-\lambda_{12}+\lambda_{23}+\lambda_{13})^2+(\lambda_{12}-\lambda_{23} +\lambda_{13})^2+(\lambda_{12}+\lambda_{23}-\lambda_{13})^2}{12 \lambda_{12}\lambda_{23}\lambda_{13}}
=\dfrac{(x+y-z)^{2}}{3xyz}.
\end{gather*}

The first Hermitian scalar curvature is given by
\begin{align*}
s_{1}=\dfrac{s}{2}-\dfrac{5}{12}\|({\rm d}F)^{-}\|^{2}+\dfrac{1}{16}\big\|N^{0}\big\|^{2}+\dfrac{1}{4}\|({\rm d}F)^{+}\|^{2}
=\dfrac{1}{3}\bigg(\dfrac{1}{x}+\dfrac{1}{y}+\dfrac{1}{z}\bigg).
\end{align*}

In this case, we have
\begin{gather*}
2s_{1}-s=\dfrac{(x+y-z)^{2}}{6xyz}.
\end{gather*}
Therefore,
\begin{gather*}
2s_{1}-s=0\Leftrightarrow z=x+y,
\end{gather*}
that is, $2s_{1}-s=0$ if, and only if, $(M,g,J_1)$ is K\"ahler.
\end{proof}

An immediate consequence of the proof of Proposition \ref{prop-su3-j1}, by analyzing the vanishing of the components $({\rm d}F)^{-}$, $({\rm d}F)^{+}$, $N^0$ is the following result:
\begin{Proposition}\label{prop-c3}
The invariant almost Hermitian structure $(g, J_{1})$ on ${\rm SU}(3)/T^2$ belongs to the Gray--Hervella class $\mathcal{W}_{3}$.%
\end{Proposition}
\begin{Remark}
It is well know that the invariant Hermitian structure $(g, J_{1})$ with $g$ satisfying $z=x+y$ is a K\"ahler structure, see \cite{wolf-gray}. It is worth to point out that the results of Propositions~\ref{prop-c1}, \ref{prop-su3-j1} and~\ref{prop-c3} are well-known, see \cite{fu2018twistor}. We include the proofs of these propositions for illustrate our methods.
\end{Remark}

For the sake of completeness we will compute the scalar curvatures $s_2$ and $s_J$ for the complex structure $J_1$. We will use the ingredients computed in the Proposition \ref{prop-su3-j1}.

\begin{Proposition}
Let us consider the flag manifold ${\rm SU}(3)/T^2$, equipped with invariant Hermitian structure $(g, J_{1})$. The second Hermitian scalar curvature $s_2$ and the $J$-scalar curvature are given by
\begin{gather*}
s_{2}(t)=\frac{(t-2) t (x+y-z)^2+x^2-6 x (y+z)+y^2-6 y z+z^2}{12 x y z},
\\
s_{J}=-\frac{x^2-6 x (y+z)+y^2-6 y z+z^2}{6 x y z}.
\end{gather*}
\end{Proposition}

\subsection[The complex projective space \protect{CP\textasciicircum{}{3}=Sp(2)/Sp(1) times U(1)}]
{The complex projective space $\boldsymbol{\mathbb{CP}^{3}={\rm Sp}(2)/{\rm Sp}(1)\times {\rm U}(1)}$}

Let us consider another 6-dimensional homogeneous space, the partial flag manifold $\mathbb{CP}^{3}={\rm Sp}(2)/{\rm Sp}(1)\times {\rm U}(1)$.
Let $\mathfrak{h}=\left(\begin{smallmatrix}
\Lambda&\\
&-\Lambda
\end{smallmatrix}\right)$, with $\Lambda=\left(\begin{smallmatrix}
x&\\
&y
\end{smallmatrix}\right)$, be the Cartan subalgebra of $\mathfrak{sp}(2)$.

One can describe precisely each one of the irreducible components $\mathfrak{m}_1, \mathfrak{m}_2$ in terms of the root space decomposition of the Lie algebra $\mathfrak{sp}(2)$. We will use the following basis of $\mathfrak{sp}(2)$: let $E_{ij}$ $(1 \leq i,j \leq 4)$ be the $4 \times 4$ matrix with $1$ in the $ij$-position and zero otherwise. We define
\begin{gather*}
X_{12} = \dfrac{\sqrt{3}}{6}(E_{12}-E_{43}), \qquad
X_{-12} = \dfrac{\sqrt{3}}{6}(E_{21} - E_{34}), \qquad
X_{12}^{+} = \dfrac{\sqrt{3}}{6}(E_{14}+E_{23}),
\\
X_{-12}^{+} = \dfrac{\sqrt{3}}{6}(E_{32}+E_{41}), \qquad
X_{11} = \dfrac{\sqrt{6}}{6}(E_{13}), \qquad
X_{-11} = \dfrac{\sqrt{6}}{6}(E_{31}).
\end{gather*}

The matrices defined above are the Weyl basis of $\mathfrak{sp}(4,\mathbb{C})$ associated to the Cartan subalgebra of diagonal matrices.

Let us consider the linear functional given by
\begin{gather*}
\lambda_{1}\colon\ \Lambda=\Diag\{x,y\}\longrightarrow x \qquad\text{and}\qquad
\lambda_{2}\colon\ \Lambda=\Diag\{x,y\}\longrightarrow y.
\end{gather*}

We have the following positive roots of $\mathfrak{sp}(2)$:
\begin{gather*}
\alpha_{1}=\lambda_{1}-\lambda_{2}, \qquad \alpha_{2}=\lambda_{1}+\lambda_{2} \qquad \text{and} \qquad \alpha_{3}=2\lambda_{1}.
\end{gather*}

The isotropy representation decomposes in the following way
\begin{align*}
\mathfrak{m}&=\mathfrak{m}_{1}\oplus \mathfrak{m}_{2}
\\
&=(\mathfrak{g}_{\alpha_{1}}\oplus \mathfrak{g}_{\alpha_{2}}) \oplus \mathfrak{g}_{\alpha_{3}}.
\end{align*}
We will denote an invariant metric $g$ by the pair of positive numbers $(x,y)$. To be consistent with Section \ref{sec-computations}, we set $\lambda_{1}=x$ and $\lambda_{2}=y$. We will consider the orthonormal basis of $TM_{\mathbb{C}}$:
\begin{gather*}
\bigg\{\dfrac{\sqrt{x}}{x\sqrt{2}}X_{12}, \, \dfrac{\sqrt{x}}{x\sqrt{2}}X_{12}^{+},\,\dfrac{\sqrt{y}}{y}X_{11},\,\dfrac{\sqrt{x}}{x\sqrt{2}}X_{-12},\, \dfrac{\sqrt{x}}{x\sqrt{2}}X_{-12}^{+},\,\dfrac{\sqrt{y}}{y}X_{-11}\bigg\}.
\end{gather*}

\begin{Proposition}[\cite{grego}]
Let us consider $\mathbb{CP}^3$, with invariant metric $g$ parametrized by $(x,y)$. Then the Riemannian scalar curvature of $(\mathbb{CP}^3,g)$ is given by
\begin{gather*}
s= \frac{2}{x}+\frac{2}{3y} - \frac{y}{6x^2}.
\end{gather*}
\end{Proposition}

The flag manifold $\mathbb{CP}^{3}$ admits two invariant almost complex structures (up to conjugation and equivalence) parametrized in the following way: the integrable structure $J_1=(+,+)$ and the non-integrable structure $J_2=(+,-)$.

\begin{Lemma}
The invariant Hermitian structure $(g, J_{2})$ does not admit $(1,2)$-triple. The $(0,3)$-triples are given by $\alpha_{1}+\alpha_{2}-\alpha_{3}=0$ and $-\alpha_{1}-\alpha_{2}+\alpha_{3}=0$.
\end{Lemma}

\begin{Proposition}\label{prop-cp3-j2}
Let us consider the flag manifold $\mathbb{CP}^3$, equipped with invariant Hermitian structure $(g, J_{2})$. Denote by $s$ the Riemannian scalar curvature and $s_1$ the first Hermitian scalar curvature. Then the invariant metric $g$ that is solution of the equation $2s_1-s=0$ is given by
 \begin{gather}
y=2x\big(\sqrt{10} +3\big).
\end{gather}
\end{Proposition}
\begin{proof}
By Proposition \ref{prop-norma}, we have $\|({\rm d}F)^{+}\|^{2}=0$,
\begin{gather*}
\|({\rm d}F)^{-}\|^{2}=\frac{(2 x+y)^2}{3 x^3},
\\
\big\|N^{0}\big\|^{2}=\frac{32 (x-y)^2}{9 x^3}
\end{gather*}
and the covariant derivative of the K\"ahler form is
\begin{gather*}
\|DF\|^2=\frac{4 x^2-4 x y+3 y^2}{3 x^3}.
\end{gather*}

The first Hermitian scalar curvature is given by
\begin{align*}
s_{1}=\dfrac{s}{2}-\dfrac{5}{12}\|({\rm d}F)^{-}\|^{2}+\dfrac{1}{16}\big\|N^{0}\big\|^{2}+\dfrac{1}{4}\|({\rm d}F)^{+}\|^{2}
=\frac{4 x^3+8 x^2 y-13 x y^2+y^3}{12 x^3 y}.
\end{align*}

Thus we obtain
\begin{gather*}
2s_{1}-s=\frac{-4 x^2-12 x y+y^2}{6 x^3}.
\end{gather*}
Therefore,
\begin{gather*}
2s_{1}-s=0\Leftrightarrow y=2x\big(\sqrt{10} +3\big).\tag*{\qed}
\end{gather*}
\renewcommand{\qed}{}
\end{proof}

An immediate consequence of the proof of Proposition \ref{prop-cp3-j2}, by analyzing the vanishing of the components $({\rm d}F)^{-}$, $({\rm d}F)^{+}$, $N^0$ is the following result:
\begin{Proposition}
With the notation above, the invariant almost Hermitian structure $(g, J_{2})$ on $\mathbb{CP}^3$ belongs to the Gray--Hervella class $\mathcal{W}_{1}\oplus \mathcal{W}_{2}$. The pair $(g, J_{2})$ belongs to the class $\mathcal{W}_{1}$ $($nearly K\"ahler$)$ if, and only if, the parameters of the metric $g$ satisfies $x=y$.
\end{Proposition}

For the sake of completeness we will compute the scalar curvatures $s_2=s_{2}(t)$ and $s_J$ for the complex structure $J_2$. We will use the ingredients computed in the Proposition \ref{prop-cp3-j2}.

\begin{Proposition}
Let us consider the flag manifold $\mathbb{CP}^3$, equipped with invariant Hermitian structure $(g, J_{2})$. The second Hermitian scalar curvature $s_2$ and the $J$-scalar curvature are given~by
\begin{gather*}
s_{2}=\frac{4 x^3+12 x^2 y-5 x y^2+y^3}{12 x^3 y},
\qquad
s_{J}=\frac{1}{6} \bigg(\frac{4 y^2}{x^3}-\frac{17 y}{x^2}+\frac{12}{x}+\frac{4}{y}\bigg).
\end{gather*}
\end{Proposition}

Let us deal now with the invariant almost Hermitian structure $(g, J_{1})$.

\begin{Lemma}
The invariant Hermitian structure $(g, J_{1})$ does not admit $(0,3)$-triple. The $(1,2)$-triples are given by $\alpha_{1}+\alpha_{2}-\alpha_{3}=0$ and $-\alpha_{1}-\alpha_{2}+\alpha_{3}=0$.
\end{Lemma}

\begin{Proposition}\label{prop-cp3-j1}
Let us consider the flag manifold $\mathbb{CP}^3$, equipped with invariant Hermitian structure $(g, J_{1})$. Denote by $s$ the Riemannian scalar curvature and $s_1$ the first Hermitian scalar curvature. Then the invariant metric $g$ that is solution of the equation $2s_1-s=0$ is given by $y=2x$. In this case, the pair $(g, J_{1})$ is a K\"ahler structure.
\end{Proposition}
\begin{proof}
We get $\|({\rm d}F)^{-}\|^{2}=\big\|N^{0}\big\|^{2}=0$ and the norm of the $(1,2)+(2,1)$-part of ${\rm d}F$ satisfies
\begin{gather*}
\|({\rm d}F)^{+}\|^{2}=\dfrac{(2 x - y)^2}{6 x^4 y}.
\end{gather*}
The first Hermitian scalar curvature is
\begin{align*}
s_{1}=\dfrac{s}{2}-\dfrac{5}{12}\|({\rm d}F)^{-}\|^{2}\!+\dfrac{1}{16}\big\|N^{0}\big\|^{2}\!+\dfrac{1}{4}\|({\rm d}F)^{+}\|^{2}
=\frac{8 x^4+24 x^3 y-2 x^2\! \big(y^2-2\big)-4 x y+y^2}{24 x^4 y}.
\end{align*}

Therefore,
\begin{gather*}
2s_{1}-s=\frac{(y-2 x)^2}{12 x^4 y}.
\end{gather*}
and
\begin{gather*}
2s_{1}-s=0\Leftrightarrow y=2x \qquad (\text{the pair}\ (g, J_1) \ \text{is K\"ahler}).\tag*{\qed}
\end{gather*}
\renewcommand{\qed}{}
\end{proof}

An immediate consequence of the proof of Proposition \ref{prop-cp3-j1}, by analyzing the vanishing of the components $({\rm d}F)^{-}$, $({\rm d}F)^{+}$, $N^0$ is the following result:
\begin{Proposition}
The invariant almost Hermitan structure $(g, J_{1})$ on $\mathbb{CP}^3$ belongs to the Gray--Hervella class $\mathcal{W}_{3}$. The invariant Hermitian structure $(g, J_{1})$ with $g$ satisfying $y=2x$ is a K\"ahler structure
\end{Proposition}

For the sake of completeness we will compute the scalar curvatures $s_2(t)$ and $s_J$ for the complex structure $J_1$. We will use the ingredients computed in the Proposition \ref{prop-cp3-j1}.

\begin{Proposition}
Let us consider the flag manifold $\mathbb{CP}^3$, equipped with invariant Hermitian structure $(g, J_{1})$. The second Hermitian scalar curvature $s_2(t)$ and the $J$-scalar curvature are given by
\begin{gather*}
s_{2}(t)=\frac{-t^2 (y-2 x)^2+2 t (y-2 x)^2+2 x^2 \big(4 x^2+12 x y-y^2\big)}{24 x^4 y},
\\
s_{J}=-\frac{y}{6 x^2}+\frac{2}{x}+\frac{2}{3 y}.
\end{gather*}
\end{Proposition}

\subsection[Full flag manifold \protect{SU(4)/T\textasciicircum{}3}]
{Full flag manifold $\boldsymbol{{\rm SU}(4)/T^3}$}

Let us consider now the 12-dimensional full flag manifold ${\rm SU}(4)/T^3$. Recall that a Cartan sub-algebra $\mathfrak{h}$ of $\mathfrak{su}(4)$ is given by
\begin{gather*}
\mathfrak{h}=\{\Diag(x_1,x_2,x_3,x_4)\colon x_1+x_2+x_3+x_4=0,\ x_1,x_2,x_3,x_4 \in\mathbb{C}\}.
\end{gather*}
The set of roots is given by $\lambda_{ij}=x_i - x_j$, and the positive roots are $\lambda_{ij}$, $i<j$.

The isotropy representation of ${\rm SU}(4)/T^3$ splits into six irreducibles components of the isotropy representation, where each component is associated to a positive root:
\begin{align*}
\mathfrak{m}&=\mathfrak{m}_{1}\oplus\mathfrak{m}_{2}\oplus \mathfrak{m}_{3}\oplus\mathfrak{m}_{4}\oplus\mathfrak{m}_{5}\oplus \mathfrak{m}_{6}
\\
&=\mathfrak{u}_{12}\oplus \mathfrak{u}_{13}\oplus \mathfrak{u}_{14}\oplus \mathfrak{u}_{23}\oplus \mathfrak{u}_{24}\oplus \mathfrak{u}_{34},
\end{align*}
where $\mathfrak{u}_{ij}=\mathfrak{su}(4)\cap (\mathfrak{g}_{ij}\oplus \mathfrak{g}_{ji})$, and $\mathfrak{g}_{ij}$ is the root space associated to the root $\lambda_{ij}$.

There exist four invariant almost complex structures, up to conjugation and equivalence:
\begin{alignat*}{3}
& J_{1}=(+,+,+,+,+,+),\qquad &&J_{2}=(-,+,+,-,+,+),&\\
&J_{3}=(+,+,+,-,+,-),\qquad &&J_{4}=(+,+,-,+,+,+).&
\end{alignat*}

The zero sum triples of ${\rm SU}(4)/T^3$ are given by
$\alpha_{12}+\alpha_{23}+\alpha_{31}=0$, $\alpha_{12}+\alpha_{24}+\alpha_{41}=0$, $\alpha_{13}+\alpha_{34}+\alpha_{41}=0$,
$\alpha_{23}+\alpha_{34}+\alpha_{42}=0$, $\alpha_{21}+\alpha_{32}+\alpha_{13}=0$, $\alpha_{21}+\alpha_{42}+\alpha_{14}=0$, $\alpha_{31}+\alpha_{43}+\alpha_{14}=0$ and $\alpha_{32}+\alpha_{43}+\alpha_{24}=0$.

An invariant metric $g$ on $ {\rm SU}(4)/T^3$ depends on six parameters. We will restrict our analysis to a family of invariant metric induced by the deformation of the normal metric ${\rm SU}(4)/T^3$ in the direction of the fibers of the fibration
\begin{gather*}
{\rm SU}(3)/T^2 \, \cdots {\rm SU}(4)/T^3 \, \mapsto \mathbb{CP}^3.
\end{gather*}
The resulting invariant metric $g$ is the 1-parameter family of metrics parametrized by
\begin{gather} \label{metric-su4}
 g=(\lambda_{12},\lambda_{13},\lambda_{14},\lambda_{23},\lambda_{24},\lambda_{34})=\big(x^{2},x^{2},1,x^{2},1,1\big).
 \end{gather}

\begin{Proposition}
The Riemannian scalar curvature $\big({\rm SU}(4)/T^3,g\big)$ is given by
\begin{gather*}
s=\frac{3}{8} \bigg({-}x^2+\frac{5}{x^2}+8\bigg).
\end{gather*}
\end{Proposition}

In the sequel we will prove Theorem \ref{teo-su4-t} stated in the Introduction. We will use extensively the results of Section \ref{sec-computations}. Since these computations are extensive and laborious we will omit some details. We will deal case-by-case for each pair of invariant almost Hermitian structure $(g,J_i)$, $i=1,\ldots, 4$, where the metric $g$ is given in equation~\eqref{metric-su4}.

\medskip\noindent
$\boldsymbol{(g, J_{1}).}$ Using the results of Section \ref{sec-computations} we have
\begin{gather*}
\big\|N^{0}\big\|^{2}=0,\\
\|({\rm d}F)^{-}\|^{2}=0,\\
\|({\rm d}F)^{+}\|^{2}=x^2+\frac{1}{3 x^2},\\
\|DF\|^{2}=x^2+\frac{1}{3 x^2}.
\end{gather*}
Therefore the invariant almost Hermitian structure $(g,J_1)$ belongs in the class $\mathcal{W}_{3}$ for all $x$.
The first Hermitian scalar curvature is
\begin{align*}
s_{1}=\dfrac{s}{2}-\dfrac{5}{12}\|({\rm d}F)^{-}\|^{2}+\dfrac{1}{16}\big\|N^{0}\big\|^{2}+\dfrac{1}{4}\|({\rm d}F)^{+}\|^{2}
=\frac{1}{48} \bigg(3 x^2+\frac{49}{x^2}+72\bigg).
\end{align*}
Therefore the equation $2s_{1}-s$ reads:
\begin{align*}
2s_{1}-s=-\dfrac{5}{6}\|({\rm d}F)^{-}\|^{2}+\dfrac{1}{8}\big\|N^{0}\big\|^{2}+\dfrac{1}{2}\|({\rm d}F)^{+}\|^{2}
=\frac{3 x^4+1}{6 x^2},
\end{align*}
and there is no $x$ such that $2s_{1}-s=0$.

\medskip\noindent
$\boldsymbol{(g,J_2).}$ Using the results of Section \ref{sec-computations} we have
\begin{gather*}
\|({\rm d}F)^{-}\|^{2}=\dfrac{3}{x^2}, \\
 \big\|N^{0}\big\|^{2}=0, \\
\|({\rm d}F)^{+}\|^{2}=x^{2}, \\
\|DF\|^{2}=x^2+\frac{1}{x^2}.
\end{gather*}
Therefore the invariant almost Hermitian structure $(g,J_2)$ belongs in the class $\mathcal{W}_{1}\oplus \mathcal{W}_{3}$ for all~$x$.
The first Hermitian scalar curvature is
\begin{align*}
s_{1}=\dfrac{s}{2}-\dfrac{5}{12}\|({\rm d}F)^{-}\|^{2}+\dfrac{1}{16}\big\|N^{0}\big\|^{2}+\dfrac{1}{4}\|({\rm d}F)^{+}\|^{2}
=\frac{x^4+24 x^2-5}{16 x^2}.
\end{align*}
Therefore the equation $2s_{1}-s$ reads:
\begin{align*}
2s_{1}-s=-\dfrac{5}{6}\|({\rm d}F)^{-}\|^{2}+\dfrac{1}{8}\big\|N^{0}\big\|^{2}+\dfrac{1}{2}\|({\rm d}F)^{+}\|^{2}
=\frac{x^4-5}{2 x^2},
\end{align*}
and $2s_{1}-s=0$ if, and only if, $x=\sqrt[4]{5}$.

\medskip\noindent
$\boldsymbol{(g,J_{3}).}$ Using the results of Section \ref{sec-computations} we have
\begin{gather*}
\|({\rm d}F)^{-}\|^{2}=\frac{\big(x^2+2\big)^2}{3 x^2}, \\
\big\|N^{0}\big\|^{2}=\frac{32 \big(x^2-1\big)^2}{9 x^2}, \\
\|({\rm d}F)^{+}\|^{2}=\frac{1}{3} \bigg(2 x^2+\frac{5}{x^2}-4\bigg), \\
\|DF\|^{2}=\frac{5 x^2}{3}+\frac{3}{x^2}-\frac{8}{3}.
\end{gather*}
Therefore the invariant almost Hermitian structure $(g,J_2)$ belongs in the class $\mathcal{W}_{1}\oplus \mathcal{W}_{2}\oplus \mathcal{W}_{3}$ for all $x\neq 1$. When $x=1$ we have $N^0=0$ and in this case $(g,J_2)\in \mathcal{W}_{1}\oplus \mathcal{W}_{3} $.
The first Hermitian scalar curvature is
\begin{align*}
s_{1}=\dfrac{s}{2}-\dfrac{5}{12}\|({\rm d}F)^{-}\|^{2}+\dfrac{1}{16}\big\|N^{0}\big\|^{2}+\dfrac{1}{4}\|({\rm d}F)^{+}\|^{2}
=\frac{3 x^4+8 x^2+49}{48 x^2}.
\end{align*}
Therefore the equation $2s_{1}-s$ reads:
\begin{align*}
2s_{1}-s=-\dfrac{5}{6}\|({\rm d}F)^{-}\|^{2}+\dfrac{1}{8}\big\|N^{0}\big\|^{2}+\dfrac{1}{2}\|({\rm d}F)^{+}\|^{2}
=\frac{1}{6} \bigg(3 x^2+\frac{1}{x^2}-16\bigg),
\end{align*}
and $2s_{1}-s=0$ if, and only if, $x=\sqrt{\frac{1}{3} \big(8-\sqrt{61}\big)}$ or $x=\sqrt{\frac{1}{3} \big(\sqrt{61}+8\big)}$.

\medskip\noindent
$\boldsymbol{(g,J_{4}).}$ Using the results of Section \ref{sec-computations} we have
\begin{gather*}
\|({\rm d}F)^{-}\|^{2}=\frac{2 \big(x^2+2\big)^2}{3 x^2},\\
\big\|N^{0}\big\|^{2}=\frac{64 \big(x^2-1\big)^2}{9 x^2}, \\
\|({\rm d}F)^{+}\|^{2}=\frac{x^4+1}{3 x^2}, \\
\|DF\|^{2}=\frac{7 x^2}{3}+\frac{3}{x^2}-\frac{8}{3}.
\end{gather*}
Therefore the invariant almost Hermitian structure $(g,J_3)$ belongs in the class $\mathcal{W}_{1}\oplus \mathcal{W}_{2}\oplus \mathcal{W}_{3}$ for all $x\neq 1$. When $x=1$ we have $N^0=0$ and in this case $(g,J_3)\in \mathcal{W}_{1}\oplus \mathcal{W}_{3} $.
The first Hermitian scalar curvature is
\begin{align*}
s_{1}=\dfrac{s}{2}-\dfrac{5}{12}\|({\rm d}F)^{-}\|^{2}+\dfrac{1}{16}\big\|N^{0}\big\|^{2}+\dfrac{1}{4}\|({\rm d}F)^{+}\|^{2}
=\frac{1}{48} \bigg(3 x^2+\frac{17}{x^2}-24\bigg).
\end{align*}
Therefore the equation $2s_{1}-s$ reads:
\begin{align*}
2s_{1}-s=-\dfrac{5}{6}\|({\rm d}F)^{-}\|^{2}+\dfrac{1}{8}\big\|N^{0}\big\|^{2}+\dfrac{1}{2}\|({\rm d}F)^{+}\|^{2}
=\frac{x^2}{2}-\frac{7}{6 x^2}-4,
\end{align*}
and $2s_{1}-s=0$ if, and only if, $x=\sqrt{\frac{1}{3} \big(\sqrt{165}+12\big)}$.

For the sake of completeness we will finish this section by computing the Hermitan scalar curvatures $s_2(t)$ and $s_J$ for almost Hermitian structures $(g, J_i)$, $i=1,\ldots,4$ on ${\rm SU}(4)/T^3$. We~summarize the computations in the next result.

\begin{Proposition}
Let us consider the flag manifold ${\rm SU}(4)/T^3$. The second Hermitian scalar curvature $s_2(t)$ and the $J$-scalar curvature $s_J$ for each invariant Hermitian structure $(g,J_i)$, $i=1,\ldots, 4$ are given by
\begin{table}[h] \setlength{\tabcolsep}{2.5pt}
\renewcommand{\arraystretch}{2.2}
\centering
\begin{tabular}{c|c|c}
\hline
\parbox[][][c]{18mm}{\vspace{1mm}\centering Invariant\\ Hermitian\\ structure\vspace{1mm}}& $s_2(t)$ & $s_J$
\\
\hline
$(g, J_1)$ & $-\dfrac{4 (t-2) t (3 x^4+1)+9 (x^4-8 x^2-5)}{48 x^2}$ & $\dfrac{3}{8} \bigg({-}x^2+\dfrac{5}{x^2}+8\bigg)$
\\
\hline
$(g, J_2)$ & $\dfrac{(-4 t^2+8 t-3) x^4+24 x^2+11}{16 x^2}$ & $-\dfrac{3 x^2}{8}-\dfrac{1}{8 x^2}+3$
\\
\hline
$(g, J_3)$ & $\dfrac{-4 t^2 (2 x^4\!-4 x^2\!+5)+8 t (2 x^4\!-4 x^2\!+5)-5 x^4+56 x^2+45}{48 x^2}$ & $\dfrac{1}{24} \bigg(7 x^2+\dfrac{45}{x^2}+8\bigg)$
\\
\hline
$(g, J_4)$ & $ \dfrac{-4 t^2 (x^4+1)+8 t (x^4+1)-x^4+40 x^2+45}{48 x^2}$ & $\dfrac{1}{24} \bigg(23 x^2+\dfrac{45}{x^2}-56\bigg)$
\\
\hline
\end{tabular}
\end{table}
\end{Proposition}

\subsection[Flag manifolds of the exceptional Lie group  \protect{$G\textunderscore{}2$}]
{Flag manifolds of the exceptional Lie group $\boldsymbol{G_{2}}$}

Let us recall some well known facts about the Lie algebra $\mathfrak{g}_{2}=\Lie(G_2)$.

Let us consider the Cartan sub-algebra $\mathfrak{h}$ of the $\mathfrak{g_{2}}$ as a subalgebra of the diagonal matrices of $\mathfrak{sl}(3)$ and let $\lambda_{i}$ be the linear functional of $\mathfrak{h}$ defined by
\begin{gather*}
\lambda_{i}\colon  \Diag\{a_{1},a_{2},a_{3}\}\longmapsto a_{i}.
\end{gather*}

The simple roots of $G_{2}$ are $\alpha_{1}=\lambda_{1}-\lambda_{2}$ and $\alpha_{2}=\lambda_{2}$. The set of positive roots is given by
\begin{gather*}
\Pi^{+}=\{\alpha_{1},\alpha_{2},\alpha_{1}+\alpha_{2},\alpha_{1}+2\alpha_{2},\alpha_{1}+3\alpha_{2},2\alpha_{1}+3\alpha_{2}\}.
\end{gather*}

The maximal root of $G_{2}$ is $\mu=2\alpha_{1}+3\alpha_{2}$.

\subsubsection[Flag of \protect{$G\textunderscore{}2$} with two isotropy components]
{Flag of $\boldsymbol{G_{2}}$ with two isotropy components}

Let us consider the 10-dimensional flag manifold $G_{2}/{\rm U}(2)$, where the isotropy group ${\rm U}(2)$ is represented by the short root. The isotropy representation of this homogeneous space decomposes into two irreducible components:
$\mathfrak{m}_{1}$ and $\mathfrak{m}_{2}$, described explicitly as
\begin{gather*}
\mathfrak{m}_{1}=R^{+}(\alpha_{1},1)=\{\alpha_{1},\alpha_{1}+\alpha_{2},\alpha_{1} +2\alpha_{2},\alpha_{1}+3\alpha_{2}\},
\\
\mathfrak{m}_{2}=R^{+}(\alpha_{1},2)=\{2\alpha_{1}+3\alpha_{2}\}.
\end{gather*}

Therefore $G_{2}/{\rm U}(2)$ admits $2$ invariant almost complex structures, up to conjugation:
\begin{gather*}
J_{1}=(+,+)\qquad \text{and}\qquad J_{2}=(+,-).
\end{gather*}

We will denote an invariant metric $g$ by the pair of positive numbers $(x,y)$. To be consistent with Section \ref{sec-computations}, we set $\lambda_{1}=x$ and $\lambda_{2}=y$. Let us consider the orthonormal basis of $TM_{\mathbb{C}}$:
\begin{gather*}
\bigg\{\dfrac{X_{\alpha_{1}}}{\sqrt{x}},\dfrac{X_{\alpha_{1}+\alpha_{2}}}{\sqrt{x}}, \dfrac{X_{\alpha_{1}+2\alpha_{2}}}{\sqrt{x}},\dfrac{X_{\alpha_{1}+3\alpha_{2}}}{\sqrt{x}}, \dfrac{X_{2\alpha_{1}+3\alpha_{2}}}{\sqrt{y}},\dfrac{X_{-\alpha_{1}}}{\sqrt{x}},
\dfrac{X_{-(\alpha_{1}+\alpha_{2})}}{\sqrt{x}},\dfrac{X_{-(\alpha_{1}+2\alpha_{2})}}{\sqrt{x}},
\\ \qquad
\dfrac{X_{-(\alpha_{1}+3\alpha_{2})}}{\sqrt{x}},\dfrac{X_{-(2\alpha_{1}+3\alpha_{2})}}{\sqrt{y}}\bigg\}.
\end{gather*}

\begin{Proposition}[\cite{grego2}]
{\sloppy
Let us consider $G_{2}/{\rm U}(2)$, with invariant metric $g$ parametrized by~$(x,y)$. Then the Riemannian scalar curvature of $(G_{2}/{\rm U}(2),g)$ is given by
\begin{gather*}
s=-\frac{y}{4 x^2}+\frac{4}{x}+\frac{1}{2 y}.
\end{gather*}}
\end{Proposition}

\begin{Lemma}
The invariant Hermitian structure $(g, J_{2})$ does not admit $(1,2)$-triple. The $(0,3)$-triples are given by $\alpha_1+( \alpha_1+3\alpha_2) - ( 2\alpha_1+3\alpha_2)=0$, $-\alpha_1 + ( -\alpha_1-3\alpha_2)+(2\alpha_1+3\alpha_2)=0$, $(\alpha_1+\alpha_2) + ( \alpha_1+2\alpha_2)-(2\alpha_1+3\alpha_2)=0$, $( -\alpha_1-\alpha_2)+ ( -\alpha_1-2\alpha_2)+(2\alpha_1+3\alpha_2)=0$.
\end{Lemma}

\begin{Proposition}\label{prop-g2-j2}
Let us consider the flag manifold $G_{2}/{\rm U}(2)$, equipped with invariant Hermitian structure $(g, J_{2})$. Denote by $s$ the Riemannian scalar curvature and $s_1$ the first Hermitian scalar curvature. Then the invariant metric $g$ that is solution of the equation $2s_1-s=0$ is given by
 \begin{gather*}
y=2x\big(\sqrt{10} +3\big).
\end{gather*}
\end{Proposition}

\begin{proof}
Note that $\|({\rm d}F)^{+}\|^{2}=0$. The squared norm of the $(0,3)+(3,0)$-part of the K\"ahler form is given by
\begin{gather*}
\|({\rm d}F)^{-}\|^{2}=\frac{(2 x+y)^2}{12 x^2 y}
\end{gather*}
and the form $N^{0}$ is
\begin{gather*}
\big\|N^{0}\big\|^{2}=\frac{8 (x-y)^2}{9 x^2 y}.
\end{gather*}

The covariant derivative of $F$ satisfies
\begin{gather*}
\|DF\|^{2}=\frac{1}{12} \bigg(\frac{3 y}{x^2}-\frac{4}{x}+\frac{4}{y}\bigg),
\end{gather*}
and the first Hermitian scalar curvature is given by
\begin{gather*}
s_{1}=\dfrac{1}{48} \bigg(\dfrac{84}{x} + \dfrac{8}{y} - \dfrac{5 y}{x^2}\bigg).
\end{gather*}

Therefore, we have
\begin{align*}
2s_{1}-s=-\dfrac{5}{6}\|({\rm d}F)^{-}\|^{2}+\dfrac{1}{8}\big\|N^{0}\big\|^{2}+\dfrac{1}{2}\|({\rm d}F)^{+}\|^{2}
=\frac{-4 x^2-12 x y+y^2}{24 x^2 y},
\end{align*}
and $2s_{1}-s=0$ if, and only if, $y=2x (\sqrt{10}+3)$.
\end{proof}

An immediate consequence of the proof of Proposition \ref{prop-g2-j2}, by analyzing the vanishing of the components $({\rm d}F)^{-}$, $({\rm d}F)^{+}$, $N^0$ is the following result:
\begin{Proposition}
With the notation above, the invariant almost Hermitan structure $(g, J_{2})$ on~$G_{2}/{\rm U}(2)$ belongs to the Gray--Hervella class $\mathcal{W}_{1}\oplus \mathcal{W}_{2}$. The pair $(g, J_{2})$ belongs to the class~$\mathcal{W}_{1}$ $($nearly K\"ahler$)$ if, and only if, the parameters of the metric $g$ satisfies $x=y$.
\end{Proposition}

For the sake of completeness we will compute the scalar curvatures $s_{2}=s_2(t)$ and $s_J$ for the complex structure $J_2$. We will use the ingredients computed in the Proposition \ref{prop-g2-j2}.

\begin{Proposition}
Let us consider the flag manifold $G_2/{\rm U}(2)$, equipped with invariant Hermitian structure $(g, J_{2})$. The second Hermitian scalar curvature $s_2$ and the $J$-scalar curvature is given by
\begin{gather*}
s_{2}=\frac{1}{48} \bigg({-}\frac{5 y}{x^2}+\frac{92}{x}+\frac{12}{y}\bigg),
\\
s_{J}=\frac{1}{12} \bigg({-}\frac{y}{x^2}+\frac{40}{x}+\frac{6}{y}\bigg).
\end{gather*}
\end{Proposition}

Let us work now with the invariant almost Hermitian structure $(g, J_{1})$.

\begin{Lemma}
The invariant Hermitian structure $(g, J_{1})$ does not admit $(0,3)$-triple. The $(1,2)$-triples are given by $\alpha_1+( \alpha_1+3\alpha_2) - ( 2\alpha_1+3\alpha_2)=0$, $-\alpha_1 + ( -\alpha_1-3\alpha_2)+(2\alpha_1+3\alpha_2)=0$, $(\alpha_1+\alpha_2) + ( \alpha_1+2\alpha_2)-(2\alpha_1+3\alpha_2)=0$, $( -\alpha_1-\alpha_2)+ ( -\alpha_1-2\alpha_2)+(2\alpha_1+3\alpha_2)=0$.
\end{Lemma}

\begin{Proposition} \label{prop-g2-j1}
Let us consider the flag manifold $G_{2}/{\rm U}(2)$, equipped with invariant Hermitian structure $(g, J_{1})$. Denote by $s$ the Riemannian scalar curvature and $s_1$ the first Hermitian scalar curvature. Then the invariant metric $g$ that is solution of the equation $2s_1-s=0$ is given by $y=2x$. In this case, the pair $(g, J_{1})$ is a K\"ahler structure.
\end{Proposition}
\begin{proof}
By Proposition \ref{prop-norma}, we obtain $\|({\rm d}F)^{-}\|^{2}=\big\|N^{0}\big\|^{2}=0$, the squared norm of the $(1,2)+(2,1)$-part of $F$ is
\begin{gather*}
\|({\rm d}F)^{+}\|^{2}=\dfrac{(2 x-y)^2}{12 x^2 y}
\end{gather*}
and the covariant derivative of $F$ is
\begin{gather*}
\|DF\|^{2}=\dfrac{(y-2 x)^2}{12 x^2 y}.
\end{gather*}

The first Hermitian scalar curvature is
\begin{gather*}
s_{1}=\dfrac{1}{48} \bigg({-}\dfrac{5 y}{x^2}+\dfrac{92}{x}+\dfrac{16}{y}\bigg).
\end{gather*}

Therefore,
\begin{align*}
2s_{1}-s=-\dfrac{5}{6}\|({\rm d}F)^{-}\|^{2}+\dfrac{1}{8}\big\|N^{0}\big\|^{2}+\dfrac{1}{2}\|({\rm d}F)^{+}\|^{2}
=\dfrac{(y-2 x)^2}{24 x^2 y}.
\end{align*}
We have $2s_{1}-s=0$ if, and only if, $y=2x$. In other words $2s_{1}-s=0$ if, and only if, $(J,\Lambda)$ is K\"ahler.
\end{proof}

An immediate consequence of the proof of Proposition \ref{prop-g2-j1}, by analyzing the vanishing of the components $({\rm d}F)^{-}$, $({\rm d}F)^{+}$, $N^0$ is the following result:
\begin{Proposition}
The invariant almost Hermitan structure $(g, J_{1})$ on $G_{2}/{\rm U}(2)$ belongs to the Gray--Hervella class $\mathcal{W}_{3}$. The invariant Hermitian structure $(g, J_{1})$ with $g$ satisfying $y=2x$ is a K\"ahler structure.
\end{Proposition}

For the sake of completeness we will compute the scalar curvatures $s_2(t)$ and $s_J$ for the complex structure $J_1$. We will use the ingredients computed in the Proposition \ref{prop-g2-j1}.

\begin{Proposition}
Let us consider the flag manifold $G_2/{\rm U}(2)$, equipped with invariant Hermitian structure $(g, J_{1})$. The second Hermitian scalar curvature $s_2(t)$ and the $J$-scalar curvature are given by
\begin{gather*}
s_{2}(t)=-\frac{4 \big(t^2-2 t-3\big) x^2-4 \big(t^2-2 t+24\big) x y+\big(t^2-2 t+6\big) y^2}{48 x^2 y},
\\
s_{J}=-\frac{y}{4 x^2}+\frac{4}{x}+\frac{1}{2 y}.
\end{gather*}
\end{Proposition}

\subsubsection[Full flag manifold \protect{$G\textunderscore{}2/T\textasciicircum{}2$}]
{Full flag manifold $\boldsymbol{G_2/T^2}$} \label{g2-full}
Let us consider now the 12-dimensional full flag manifold $G_2/T^2$. Since on flag manifolds the irreducible components of the isotropy representation are the real root space, we have 6 isotropy components given by
\begin{align*}
\mathfrak{m}&=\mathfrak{m_{1}}\oplus\mathfrak{m_{2}}\oplus\mathfrak{m_{3}} \oplus\mathfrak{m_{4}}\oplus\mathfrak{m_{5}}\oplus\mathfrak{m_{6}}
\\
&=\mathfrak{g}_{\alpha_{1}}\oplus\mathfrak{g}_{\alpha_{2}}\oplus\mathfrak{g}_{\alpha_{1} +\alpha_{2}}\oplus\mathfrak{g}_{\alpha_{1}+2\alpha_{2}}\oplus \mathfrak{g}_{\alpha_{1}+3\alpha_{2}}\oplus\mathfrak{g}_{2\alpha_{1}+3\alpha_{2}}.
\end{align*}

There exist 32 invariant almost complex structure on $G_2/T^2$, up to conjugation and equivalence. These invariant almost complex structures are parametrized by a set of 6 signals, as described in Table \ref{tab:long}.

The zero sum triples of $G_2/T^2$ (up to sign) are given by
\begin{gather} \label{tripla-soma-zero-g2}
\begin{cases}
 \alpha_{1}+\alpha_{2}-(\alpha_{1}+\alpha_{2})=0,
 \\
 \alpha_{2}+(\alpha_{1}+\alpha_{2})-(\alpha_{1}+2\alpha_{2})=0,
 \\
 \alpha_{2}+(\alpha_{1}+2\alpha_{2})-(\alpha_{1}+3\alpha_{2})=0,
 \\
 \alpha_{1}+(\alpha_{1}+3\alpha_{2})-(2\alpha_{1}+3\alpha_{2})=0,
 \\
 (\alpha_{1}+\alpha_{2})+(\alpha_{1}+2\alpha_{2})-(2\alpha_{1}+3\alpha_{2})=0.
\end{cases}
\end{gather}

An invariant metric $g$ on $G_2/T^2$ depends on six parameters. We will restrict our analysis to a family of invariant metric induced by the deformation of the normal metric $G_2/T^2$ in the direction of the fibers of the fibration
\begin{gather}
S^2\times S^2  \cdots G_2/T^2  \mapsto G_2/{\rm SO}(4).
\end{gather}
The resulting invariant metric $g$ is the 1-parameter family of metrics parametrized by
\begin{gather} \label{metric-g2}
g=(\lambda_{\alpha_{1}},\lambda_{\alpha_{2}},\lambda_{\alpha_{1}+\alpha_{2}}, \lambda_{\alpha_{1}+2\alpha_{2}},\lambda_{\alpha_{1}+3\alpha_{2}},\lambda_{2\alpha_{1}+3\alpha_{2}})= \big(1,1,x^{2},1,x^{2},1\big).
 \end{gather}

One can compute the Riemannian scalar curvature by using the Ricci curvature of an invariant metric on $G_2/T^2$ computed in \cite{AC-g2}.
\begin{Proposition}
The Riemannian scalar curvature of $\big(G_2/T^2,g\big)$, where the invariant metric~$g$ is defined in \eqref{metric-g2}, is given by
\begin{gather*}
s= \frac{2+12x^2-2x^4}{3x^2}.
\end{gather*}
\end{Proposition}

In the sequel we will summarize the computation used in the proof of Theorem \ref{teo-g2-t2} stated in the Introduction. We will deal case-by-case for each pair of invariant almost Hermitian structure $(g,J_i)$, $i=1,\ldots, 32$, where the metric $g$ is given in equation~\eqref{metric-g2}. Since the details of the proof are similar to the previous section we will omit them. The solutions for the equation $2s_1-s=0$ and all other information about the invariant almost Hermitian structures $(g,J_i)$ concerning to this question are summarized in Table \ref{tab:long}. We remark that we list just $(0,3)$-triple of the pair $(g,J_i)$. The set of the corresponding $(1,2)$-triples is the complementary set of the $(0,3)$-triples inside the set of all zero sum triples listed in \eqref{tripla-soma-zero-g2}.

\subsection*{Acknowledgements}
L.~Grama is partially supported by 2018/13481-0; 2021/04003-0 (FAPESP) and 305036/2019-0 (CNPq).

 {\begin{landscape}
\small
\begin{center}
\setlength{\tabcolsep}{1.9pt}
\renewcommand{\arraystretch}{1.5}
\begin{longtable}{c|c|c|c|c|c|c|c|c|c}
\caption{Full flag manifold $G_2/T^2$.} \label{tab:long}
\\
\hline
\parbox[][][c]{18mm}{\vspace{1mm}\centering Invariant\\ Hermitian\\ structure\vspace{1mm}}
&\parbox[c]{25mm}{\vspace{1mm}\centering Parametrization\\ of $J_i$\vspace{1mm}}
&\parbox[][][c]{22mm}{\vspace{1mm}\centering $(0,3)$-triples\\ (up to sign)\vspace{1mm}}
&\multicolumn{1}{c|}{$({\rm d}F)^{-}$}
&\multicolumn{1}{c|}{$N^{0}$}
&\multicolumn{1}{c|}{$({\rm d}F)^{+}$}
&\multicolumn{1}{c|}{$\alpha_{F}$}
&\multicolumn{1}{c|}{Class}
&\multicolumn{1}{c|}{Remark}
&\multicolumn{1}{c}{$2s_{1}-s=0$}
\\
\endfirsthead
\multicolumn{10}{c}%
{{\bf \tablename\ \thetable{}.} Continued from previous page}
\\
\hline
\parbox[][][c]{18mm}{\vspace{1mm}\centering Invariant\\ Hermitian\\ structure\vspace{1mm}}
&\parbox[c]{25mm}{\vspace{1mm}\centering Parametrization\\ of $J_i$\vspace{1mm}}
&\parbox[][][c]{22mm}{\vspace{1mm}\centering $(0,3)$-triples\\ (up to sign)\vspace{1mm}}
&\multicolumn{1}{c|}{$({\rm d}F)^{-}$}
&\multicolumn{1}{c|}{$N^{0}$}
&\multicolumn{1}{c|}{$({\rm d}F)^{+}$}
&\multicolumn{1}{c|}{$\alpha_{F}$}
&\multicolumn{1}{c|}{Class}
&\multicolumn{1}{c|}{Remark}
&\multicolumn{1}{c}{$2s_{1}-s=0$}
\\
\hline
\endhead
\hline\multicolumn{10}{r}{{Continued on next page}} \\
\endfoot
\hline
\endlastfoot
\hline
$(g,J_{1})$ & $(+,+,+,+,+,+)$ & $0$ & $0$ & $0$ & $\frac{1}{3} (5 x^2+\frac{8}{x^2}-8)$ & 0 & $\mathcal{W}_{3}$ & $\forall x$ & $\nexists x$
\\
\hline
$(g,J_{2})$ & $(-,+,+,+,+,+)$ & $0$ & 0 & $0$ & $\frac{1}{3} (5 x^2+\frac{8}{x^2}-8)$ & 0 & $\mathcal{W}_{3}$ & $\forall x$ & $\nexists x$
\\
\hline
$(g,J_{3})$ & $(+,-,+,+,+,+)$& $0$ & 0 & 0 & $\frac{1}{3} (5 x^2+\frac{4}{x^2}-4) $ & 0 & $\mathcal{W}_{3}$ & $\forall x$ & $\nexists x$
\\
\hline
$(g,J_{4})$ & $(+,+,-,+,+,+)$ & $\alpha_{1}+\alpha_{2}-(\alpha_{1}+\alpha_{2})=0$ & $\frac{(x^2+2)^2}{3 x^2} $ & $\frac{32 (x^2-1)^2}{9 x^2}$ & $ \frac{4 (x^4-x^2+1)}{3 x^2}$ & 0 & $\mathcal{W}_{1}\oplus \mathcal{W}_{2}\oplus\mathcal{W}_{3}$ & $x\neq1$ & $x=\frac{4}{\sqrt{5}}$
\\
& & & & & & & $\mathcal{W}_{1}\oplus\mathcal{W}_{3}$ & $x=1$ & $\nexists x$
\\
\hline
$(g,J_{5})$ & $(+,+,+,-,+,+)$ & $\alpha_{2}+(\alpha_{1}+\alpha_{2})-(\alpha_{1}+2\alpha_{2})=0$ & $\frac{(x^2+2)^2}{3 x^2}$ & $\frac{32 (x^2-1)^2}{9 x^2}$ & $\frac{4 (x^4-2 x^2+2)}{3 x^2}$ & 0 & $\mathcal{W}_{1}\oplus \mathcal{W}_{2}\oplus\mathcal{W}_{3}$ & $x\neq1$ & $x=\frac{23}{50}$ or $x=\frac{39}{20}$
\\
& & & & & & & $\mathcal{W}_{1}\oplus\mathcal{W}_{3}$ & $x=1$ & $\nexists x$
\\
\hline
$(g,J_{6})$ &$(+,+,+,+,-,+)$ & $\alpha_{2}+(\alpha_{1}+2\alpha_{2})-(\alpha_{1}+3\alpha_{2})=0$& $\frac{(x^2+2)^2}{3 x^2}$ & $\frac{32 (x^2-1)^2}{9 x^2}$ & $\frac{4 (x^4-x^2+1)}{3 x^2}$ & 0 & $\mathcal{W}_{1}\oplus \mathcal{W}_{2}\oplus\mathcal{W}_{3}$ & $x\neq1$ & $x=\frac{4}{\sqrt{5}}$
\\
& & & & & & & $\mathcal{W}_{1}\oplus\mathcal{W}_{3}$ & $x=1$ & $\nexists x$
\\
\hline
$(g,J_{7})$ & $(+,+,+,+,+,-)$ & $\alpha_{1}+(\alpha_{1}+3\alpha_{2})-(2\alpha_{1}+3\alpha_{2})=0$ & $\frac{2 (x^2+2)^2}{3 x^2}$ & $\frac{64 (x^2-1)^2}{9 x^2}$ & $x^2+\frac{8}{3 x^2}-\frac{8}{3}$ & 0 & $\mathcal{W}_{1}\oplus \mathcal{W}_{2}\oplus\mathcal{W}_{3}$ & $x\neq1$ & $x=\frac{4\sqrt{2}}{\sqrt{5}}$ \\
& & $(\alpha_{1}\!+\!\alpha_{2})\!+\!(\alpha_{1}\!+\!2\alpha_{2})\!-\!(2\alpha_{1}\!+\!3\alpha_{2})=0$ & & & & & $\mathcal{W}_{1}\oplus\mathcal{W}_{3}$ & $x=1$ & $\nexists x$
\\
\hline
$(g,J_{8})$ &$(+,-,-,+,+,+)$ & $\alpha_{2}+(\alpha_{1}+\alpha_{2})-(\alpha_{1}+2\alpha_{2})=0$ & $\frac{(x^2+2)^2}{3 x^2}$ & $\frac{32 (x^2-1)^2}{9 x^2}$ & $\frac{4 x^2}{3}$ & 0 & $\mathcal{W}_{1}\oplus \mathcal{W}_{2}\oplus\mathcal{W}_{3}$ & $x\neq1$ & $x=\frac{41}{25}$
\\
& & & & & & & $\mathcal{W}_{1}\oplus\mathcal{W}_{3}$ & $x=1$ & $\nexists x$
\\
\hline
$(g,J_{9})$ & $(+,-,+,-,+,+)$ & $\alpha_{2}+(\alpha_{1}+2\alpha_{2})-(\alpha_{1}+3\alpha_{2})=0$ & $\frac{(x^2+2)^2}{3 x^2}$ & $\frac{32 (x^2-1)^2}{9 x^2}$ & $\frac{4 (x^4-x^2+1)}{3 x^2}$ & 0 & $\mathcal{W}_{1}\oplus \mathcal{W}_{2}\oplus\mathcal{W}_{3}$ & $x\neq1$ & $x=\frac{4}{\sqrt{5}}$
\\
& & & & & & & $\mathcal{W}_{1}\oplus\mathcal{W}_{3}$ & $x=1$ & $\nexists x$
\\
\hline
$(g,J_{10})$& $(+,-,+,+,-,+)$ & $0$ & 0 & 0 & $\frac{1}{3} (5 x^2+\frac{4}{x^2}-4)$ & 0 & $\mathcal{W}_{3}$ & $\forall x$ & $\nexists x$
\\
\hline
$(g,J_{11})$ & $(+,-,+,+,+,-)$ & $\alpha_{1}+(\alpha_{1}+3\alpha_{2})-(2\alpha_{1}+3\alpha_{2})=0$ & $\frac{2 (x^2+2)^2}{3 x^2}$ & $\frac{64 (x^2-1)^2}{9 x^2}$ & $x^2+\frac{4}{3 x^2}-\frac{4}{3}$ & 0 & $\mathcal{W}_{1}\oplus \mathcal{W}_{2}\oplus\mathcal{W}_{3}$ & $x\neq1$ & $x=\frac{12}{5}$ \\
& & $(\alpha_{1}\!+\!\alpha_{2})\!+\!(\alpha_{1}\!+\!2\alpha_{2})\!-\!(2\alpha_{1}\!+\!3\alpha_{2})=0$ & & & & & $\mathcal{W}_{1}\oplus\mathcal{W}_{3}$ & $x=1$ & $\nexists x$
\\
\hline
$(g,J_{12})$ & $(+,+,-,-,+,+)$ & $\alpha_{1}+(\alpha_{2})-(\alpha_{1}+\alpha_{2})=0$, & $\frac{2 (x^2+2)^2}{3 x^2}$ & $\frac{64 (-1 + x^2)^2}{9 x^2}$ & $x^2+\frac{4}{3 x^2}-\frac{4}{3}$ & 0 & $\mathcal{W}_{1}\oplus \mathcal{W}_{2}\oplus\mathcal{W}_{3}$ & $x\neq1$ & $x=\frac{12}{5}$
\\
& & $(\alpha_{1}\!+\!\alpha_{2})\!+\!(\alpha_{1}\!+\!2\alpha_{2})\!-\!(2\alpha_{1}\!+\!3\alpha_{2})=0$ & & & & & $\mathcal{W}_{1}\oplus\mathcal{W}_{3}$ & $x=1$ & $\nexists x$
\\
\hline
$(g,J_{13})$ & $(+,+,-,+,-,+)$ & $\alpha_{1}+(\alpha_{2})-(\alpha_{1}+\alpha_{2})=0$, & $\frac{2 (x^2+2)^2}{3 x^2}$ & $\frac{64 (x^2-1)^2}{9 x^2}$ & $x^2$ & 0 & $\mathcal{W}_{1}\oplus \mathcal{W}_{2}\oplus\mathcal{W}_{3}$ & $x\neq1$ & $x=\frac{113}{50}$
\\
& & $\alpha_{2}+(\alpha_{1}+2\alpha_{2})-(\alpha_{1}+3\alpha_{2})=0$ & & & & & $\mathcal{W}_{1}\oplus\mathcal{W}_{3}$ & $x=1$ & $\nexists x$
\\
\hline
$(g,J_{14})$ & $(+,+,-,+,+,-)$ & $\alpha_{1}+(\alpha_{1}+3\alpha_{2})-(2\alpha_{1}+3\alpha_{2})=0$, & $\frac{2 (x^2+2)^2}{3 x^2}$ & $\frac{64 (x^2-1)^2}{9 x^2}$ & $x^2+\frac{8}{3 x^2}-\frac{8}{3}$ & 0 & $\mathcal{W}_{1}\oplus \mathcal{W}_{2}\oplus\mathcal{W}_{3}$ & $x\neq1$ & $x=\frac{4\sqrt{2}}{\sqrt{5}}$ \\
& & $\alpha_{1}+\alpha_{2}-(\alpha_{1}+\alpha_{2})=0$ & & & & & $\mathcal{W}_{1}\oplus\mathcal{W}_{3}$ & $x=1$ & $\nexists x$
\\
\hline
$(g,J_{15})$ & $(+,+,+,-,-,+)$ & $\alpha_{2}+(\alpha_{1}+\alpha_{2})-(\alpha_{1}+2\alpha_{2})=0$ & $\frac{(x^2+2)^2}{3 x^2}$ & $\frac{32 (x^2-1)^2}{9 x^2}$ & $\frac{4 (x^4-2 x^2+2)}{3 x^2}$ & 0 & $\mathcal{W}_{1}\oplus \mathcal{W}_{2}\oplus\mathcal{W}_{3}$ & $x\neq1$ & $x=\frac{23}{50}$ or $x=\frac{39}{20}$
\\
& & & & & & & $\mathcal{W}_{1}\oplus\mathcal{W}_{3}$ & $x=1$ & $\nexists x$
\\
\hline
$(g,J_{16})$ & $(+,+,+,-,+,-)$ & $\alpha_{1}+(\alpha_{1}+3\alpha_{2})-(2\alpha_{1}+3\alpha_{2})=0$, & $\frac{2 (x^2+2)^2}{3 x^2}$ & $\frac{64 (x^2-1)^2}{9 x^2}$ & $x^2+\frac{4}{3 x^2}-\frac{4}{3}$ & 0 & $\mathcal{W}_{1}\oplus \mathcal{W}_{2}\oplus\mathcal{W}_{3}$ & $x\neq1$ & $x=\frac{12}{5}$
\\
& & $\alpha_{2}+(\alpha_{1}+\alpha_{2})-(\alpha_{1}+2\alpha_{2})=0$ & & & & & $\mathcal{W}_{1}\oplus\mathcal{W}_{3}$ & $x=1$ & $\nexists x$
\\
\hline
$(g,J_{17})$ & $(+,+,+,+,-,-)$ & $\alpha_{2}+(\alpha_{1}+2\alpha_{2})-(\alpha_{1}+3\alpha_{2})=0$, & $\frac{2 (x^2+2)^2}{3 x^2}$ & $\frac{64 (x^2-1)^2}{9 x^2}$ & $x^2+\frac{8}{3 x^2}-\frac{8}{3}$ & 0 & $\mathcal{W}_{1}\oplus \mathcal{W}_{2}\oplus\mathcal{W}_{3}$ & $x\neq1$ & $x=\frac{4\sqrt{2}}{\sqrt{5}}$ \\
& & $(\alpha_{1}\!+\!\alpha_{2})\!+\!(\alpha_{1}\!+\!2\alpha_{2})\!-\!(2\alpha_{1}\!+\!3\alpha_{2})=0$ & & & & & $\mathcal{W}_{1}\oplus\mathcal{W}_{3}$ & $x=1$ & $\nexists x$
\\
\hline
$(g,J_{18})$ & $(+,-,-,-,+,+)$ & $\alpha_{2}+(\alpha_{1}+2\alpha_{2})-(\alpha_{1}+3\alpha_{2})=0$, & $\frac{2 (x^2+2)^2}{3 x^2}$ & $\frac{64 (x^2-1)^2}{9 x^2}$ & $x^2$ & 0 & $\mathcal{W}_{1}\oplus \mathcal{W}_{2}\oplus\mathcal{W}_{3}$ & $x\neq1$ & $x=\frac{113}{50}$
\\
& & $(\alpha_{1}\!+\!\alpha_{2})\!+\!(\alpha_{1}\!+\!2\alpha_{2})\!-\!(2\alpha_{1}\!+\!3\alpha_{2})=0$ & & & & & $\mathcal{W}_{1}\oplus\mathcal{W}_{3}$ & $x=1$ & $\nexists x$
\\
\hline
$(g,J_{19})$ & $(+,-,-,+,-,+)$ & $\alpha_{2}+(\alpha_{1}+\alpha_{2})-(\alpha_{1}+2\alpha_{2})=0$ & $\frac{(x^2+2)^2}{3 x^2}$ & $\frac{32 (x^2-1)^2}{9 x^2}$ & $\frac{4 x^2}{3}$ & 0 & $\mathcal{W}_{1}\oplus \mathcal{W}_{2}\oplus\mathcal{W}_{3}$ & $x\neq1$ & $x=\frac{41}{25}$\\
& & & & & & & $\mathcal{W}_{1}\oplus\mathcal{W}_{3}$ & $x=1$ & $\nexists x$
\\
\hline
$(g, J_{20})$ & $(+,-,-,+,+,-)$ & $\alpha_{2}+(\alpha_{1}+\alpha_{2})-(\alpha_{1}+2\alpha_{2})=0$, & $\frac{2 (x^2+2)^2}{3 x^2}$ & $\frac{64 (x^2-1)^2}{9 x^2}$ & $x^2+\frac{4}{3 x^2}-\frac{4}{3}$ & 0 & $\mathcal{W}_{1}\oplus \mathcal{W}_{2}\oplus\mathcal{W}_{3}$ & $x\neq1$ & $x=\frac{12}{5}$ \\
& & $\alpha_{1}+(\alpha_{1}+3\alpha_{2})-(2\alpha_{1}+3\alpha_{2})=0$ & & & & & $\mathcal{W}_{1}\oplus\mathcal{W}_{3}$ & $x=1$ & $\nexists x$
\\
\hline
$(g, J_{21})$ & $(+,-,+,-,-,+)$ & $0$ & 0 & 0 & $\frac{1}{3} (5 x^2+\frac{8}{x^2}-8)$ & 0 & $\mathcal{W}_{3}$ & $\forall x$ & $\nexists x$
\\
\hline
$(g, J_{22})$ &$(+,-,+,-,+,-)$ & $\alpha_{1}+(\alpha_{1}+3\alpha_{2})-(2\alpha_{1}+3\alpha_{2})=0$ & $\frac{2 (x^2+2)^2}{3 x^2}$ & $\frac{64 (x^2-1)^2}{9 x^2}$ & $x^2$ & 0 & $\mathcal{W}_{1}\oplus \mathcal{W}_{2}\oplus\mathcal{W}_{3}$ & $x\neq1$ & $x=\frac{113}{50}$
\\
& & $\alpha_{2}+(\alpha_{1}+2\alpha_{2})-(\alpha_{1}+3\alpha_{2})=0$ & & & & & $\mathcal{W}_{1}\oplus\mathcal{W}_{3}$ & $x=1$ & $\nexists x$
\\
\hline
$(g,J_{23})$ & $(+,-,+,+,-,-)$ & $(\alpha_{1}\!+\!\alpha_{2})\!+\!(\alpha_{1}\!+\!2\alpha_{2})\!-\!(2\alpha_{1}\!+\!3\alpha_{2})=0$ & $\frac{(x^2+2)^2}{3 x^2}$ & $\frac{32 (x^2-1)^2}{9 x^2}$ & $\frac{4 (x^4-2 x^2+2)}{3 x^2}$ & 0 & $\mathcal{W}_{1}\oplus \mathcal{W}_{2}\oplus\mathcal{W}_{3}$ & $x\neq1$ & $x=\frac{23}{50}$ or $x=\frac{39}{20}$
\\
& & & & & & & $\mathcal{W}_{1}\oplus\mathcal{W}_{3}$ & $x=1$ & $\nexists x$
\\
\hline
$(g,J_{24})$ & $(+,+,-,-,-,+)$ & $\alpha_{1}+(\alpha_{2})-(\alpha_{1}+\alpha_{2})=0$,& $\frac{2 (x^2+2)^2}{3 x^2}$ & $\frac{64 (x^2-1)^2}{9 x^2}$ & $x^2+\frac{4}{3 x^2}-\frac{4}{3}$ & 0 & $\mathcal{W}_{1}\oplus \mathcal{W}_{2}\oplus\mathcal{W}_{3}$ & $x\neq1$ & $x=\frac{12}{5}$
\\
& & $(\alpha_{1}\!+\!\alpha_{2})\!+\!(\alpha_{1}\!+\!2\alpha_{2})\!-\!(2\alpha_{1}\!+\!3\alpha_{2})=0$ & & & & & $\mathcal{W}_{1}\oplus\mathcal{W}_{3}$ & $x=1$ & $\nexists x$
\\
\hline
$(g, J_{25})$ & $(+,+,-,+,-,-)$ & $\alpha_{1}+(\alpha_{2})-(\alpha_{1}+\alpha_{2})=0$, & $\frac{2 (x^2+2)^2}{3 x^2}$ & $\frac{64 (x^2-1)^2}{9 x^2}$ & $x^2+\frac{8}{3 x^2}-\frac{8}{3}$ & 0 & $\mathcal{W}_{1}\oplus \mathcal{W}_{2}\oplus\mathcal{W}_{3}$ & $x\neq1$ & $x=\frac{4\sqrt{2}}{\sqrt{5}}$
\\
& & $\alpha_{2}+(\alpha_{1}+2\alpha_{2})-(\alpha_{1}+3\alpha_{2})=0$ & & & & & $\mathcal{W}_{1}\oplus\mathcal{W}_{3}$ & $x=1$ & $\nexists x$
\\
\hline
$(g, J_{26})$ & $(+,+,+,-,-,-)$ & $\alpha_{2}+(\alpha_{1}+\alpha_{2})-(\alpha_{1}+2\alpha_{2})=0$ & $\frac{(x^2+2)^2}{3 x^2}$ & $\frac{32 (x^2-1)^2}{9 x^2}$ & $\frac{4 (x^4-2 x^2+2)}{3 x^2}$ & 0 & $\mathcal{W}_{1}\oplus \mathcal{W}_{2}\oplus\mathcal{W}_{3}$ & $x\neq1$ & $x=\frac{23}{50}$ or $x=\frac{39}{20}$
\\
& & & & & & & $\mathcal{W}_{1}\oplus\mathcal{W}_{3}$ & $x=1$ & $\nexists x$
\\
\hline
$(g, J_{27})$ & $(+,-,-,-,-,+)$ & $(\alpha_{1}\!+\!\alpha_{2})\!+\!(\alpha_{1}\!+\!2\alpha_{2})\!-\!(2\alpha_{1}\!+\!3\alpha_{2})=0$ & $\frac{(x^2+2)^2}{3 x^2}$ & $\frac{32 (x^2-1)^2}{9 x^2}$ & $\frac{4 (x^4-x^2+1)}{3 x^2}$ & 0 & $\mathcal{W}_{1}\oplus \mathcal{W}_{2}\oplus\mathcal{W}_{3}$ & $x\neq1$ & $x=\frac{4}{\sqrt{5}}$
\\
& & & & & & & $\mathcal{W}_{1}\oplus\mathcal{W}_{3}$ & $x=1$ & $\nexists x$
\\
\hline
$(g,J_{28})$ & $(+,-,-,-,+,-)$ & $\alpha_{1}+(\alpha_{1}+3\alpha_{2})-(2\alpha_{1}+3\alpha_{2})=0$, & $\frac{2 (x^2+2)^2}{3 x^2}$ & $\frac{64 (x^2-1)^2}{9 x^2}$ & $x^2$ & 0 & $\mathcal{W}_{1}\oplus \mathcal{W}_{2}\oplus\mathcal{W}_{3}$ & $x\neq1$ & $x=\frac{113}{50}$
\\
& & $\alpha_{2}+(\alpha_{1}+2\alpha_{2})-(\alpha_{1}+3\alpha_{2})=0$ & & & & & $\mathcal{W}_{1}\oplus\mathcal{W}_{3}$ & $x=1$ & $\nexists x$
\\
\hline
$(g, J_{29})$ & $(+,-,-,+,-,-)$ & $\alpha_{2}+(\alpha_{1}+\alpha_{2})-(\alpha_{1}+2\alpha_{2})=0$ & $\frac{(x^2+2)^2}{3 x^2}$ & $\frac{32 (x^2-1)^2}{9 x^2}$ & $\frac{4 (x^4-2 x^2+2)}{3 x^2}$ & 0 & $\mathcal{W}_{1}\oplus \mathcal{W}_{2}\oplus\mathcal{W}_{3}$ & $x\neq1$ & $x=\frac{23}{50}$ or $x=\frac{39}{20}$
\\
& & & & & & & $\mathcal{W}_{1}\oplus\mathcal{W}_{3}$ & $x=1$ & $\nexists x$
\\
\hline
$(g, J_{30})$ & $(+,-,+,-,-,-)$ & $0$ & 0 & 0 & $\frac{1}{3} (5 x^2+\frac{8}{x^2}-8)$ & 0 & $\mathcal{W}_{3}$ & $\forall x$ & $\nexists x$
\\
\hline
$(g, J_{31})$ & $(+,+,-,-,-,-)$ & $\alpha_{1}+(\alpha_{2})-(\alpha_{1}+\alpha_{2})=0$ & $\frac{(x^2+2)^2}{3 x^2}$ & $\frac{32 (x^2-1)^2}{9 x^2}$ & $\frac{4 (x^4-2 x^2+2)}{3 x^2}$ & 0 & $\mathcal{W}_{1}\oplus \mathcal{W}_{2}\oplus\mathcal{W}_{3}$ & $x\neq1$ & $x=\frac{23}{50}$ or $x=\frac{39}{20}$
\\
& & & & & & & $\mathcal{W}_{1}\oplus\mathcal{W}_{3}$ & $x=1$ & $\nexists x$
\\
\hline
$(g, J_{32})$ & $(+,+,-,-,+,-)$ & $\alpha_{1}+(\alpha_{2})-(\alpha_{1}+\alpha_{2})=0$, & $\frac{2 (x^2+2)^2}{3 x^2}$ & $\frac{64 (x^2-1)^2}{9 x^2}$ & $x^2+\frac{4}{3 x^2}-\frac{4}{3}$ & 0 & $\mathcal{W}_{1}\oplus \mathcal{W}_{2}\oplus\mathcal{W}_{3}$ & $x\neq1$ & $x=\frac{12}{5}$
\\
& & $\alpha_{1}+(\alpha_{1}+3\alpha_{2})-(2\alpha_{1}+3\alpha_{2})=0$ & & & & & $\mathcal{W}_{1}\oplus\mathcal{W}_{3}$ & $x=1$ & $\nexists x$
\end{longtable}
\end{center}
\normalfont
\end{landscape}}

\pdfbookmark[1]{References}{ref}
\LastPageEnding

\end{document}